\documentclass[a4paper,10pt]{article}

\usepackage[utf8]{inputenc}
\usepackage[english]{babel}

\usepackage{mathtools}
\usepackage{amsmath}
\numberwithin{equation}{section} 
\usepackage{amssymb}
\usepackage{amsthm}
\usepackage{cases}
\usepackage{stix}
\usepackage{pgf,tikz}
\usetikzlibrary{arrows}
\usepackage{color}
\usepackage[symbol]{footmisc}
\renewcommand{\thefootnote}{\fnsymbol{footnote}}

\usepackage[pdftex,colorlinks]{hyperref}
\hypersetup{                    %
	      allcolors = blue}

\usepackage[textwidth=6in,textheight=9.25in, paperwidth=7.5in,paperheight=11.25in, inner=2.5cm, outer=2.5cm]{geometry} 
\allowdisplaybreaks[2]

\allowdisplaybreaks
\newcommand{\R}{\mathbb{R}}
\newcommand{\N}{\mathbb{N}}
\newcommand{\Z}{\mathbb{Z}}

\newcommand{\T}{\mathbb T}
\newcommand{\ccinf}{C^{\infty}}

\newcommand{\loc}{\mathrm{loc}}
\renewcommand{\cal}[1]{\mathcal C^{#1,\alpha}}
\newcommand{\cbe}[1]{\mathcal C^{#1,\beta}}
\newcommand{\I}{[0,\tau]}
\newcommand{\linf}{L^{\infty}}

\newcommand{\parder}[2]{\frac{\partial #1}{\partial #2}}
\newcommand{\pardder}[2]{\frac{\partial^2 #1}{\partial #2^2}}
\newcommand{\pd}[3]{\frac{\partial^2 #3}{\partial#1\partial#2}}
\newcommand{\pdx}[1]{\frac{\partial #1}{\partial x_1}}
\newcommand{\pdy}[1]{\frac{\partial #1}{\partial x_2}}
\newcommand{\pdz}[1]{\frac{\partial #1}{\partial x_3}}
\newcommand{\pdxx}[1]{\pardder{#1}{x_1}}
\newcommand{\pdyy}[1]{\pardder{#1}{x_2}}

\newcommand{\pdxy}[1]{\pd{x_1}{x_2}{#1}}
\newcommand{\pdxz}[1]{\pd{x_1}{x_3}{#1}}
\newcommand{\pdyz}[1]{\pd{x_2}{x_3}{#1}}
\newcommand{\pdX}[1]{\frac{\partial #1}{\partial X_1}}
\newcommand{\pdY}[1]{\frac{\partial #1}{\partial X_2}}
\newcommand{\pdZ}[1]{\frac{\partial #1}{\partial X_3}}
\newcommand{\pdXX}[1]{\pardder{#1}{X_1}}
\newcommand{\pdYY}[1]{\pardder{#1}{X_2}}
\newcommand{\pdZZ}[1]{\pardder{#1}{X_3}}
\newcommand{\pdXY}[1]{\pd{X_1}{X_2}{#1}}

\newcommand{\dz}{\partial_{x_3}}
\newcommand{\dt}{\partial_t}
\newcommand{\dd}{\partial}
\newcommand{\dxx}{\partial_{x_1 x_1}}
\newcommand{\dyy}{\partial_{x_2 x_2}}

\newcommand{\dxy}{\partial_{x_1 x_2}}

\newcommand{\dZ}{\partial_{X_3}}

\newcommand{\dZZ}{\partial_{X_3 X_3}}

\DeclareMathOperator{\supp}{supp}                  
\DeclareMathOperator{\Tr}{Tr}                      
\renewcommand{\bar}[1]{\overline{#1}}
\DeclareMathOperator{\id}{id}

\let\div=\undefined\DeclareMathOperator*{\div}{div}
\DeclareMathOperator{\Lip}{Lip}
\DeclareMathOperator{\diam}{diam}

\newcommand{\Top}{\mathcal T}

\newcommand{\al}{\alpha}
\newcommand{\be}{\beta}

\newcommand{\Dl}{\Delta}

\newcommand{\g}{\gamma}
\newcommand{\G}{\Gamma}
\newcommand{\ld}{\lambda}

\newcommand{\s}{\sigma}

\newcommand{\ta}{\theta}
\newcommand{\Ta}{\Theta}

\renewcommand{\o}{\omega}
\renewcommand{\O}{\Omega}

\newcommand{\wto}{\rightharpoonup}                 

\definecolor{gr}{rgb}   {0.,   0.69,   0.23 }
\definecolor{bl}{rgb}   {0.,   0.5,   1. }
\definecolor{mg}{rgb}   {0.85,  0.,    0.85}
\definecolor{gy}{rgb}   {0.8,  0.8,   0.8}
\definecolor{yl}{rgb}   {0.8,  0.7,   0.}
\definecolor{or}{rgb}  {0.7,0.2,0.2}

\newcommand{\mres}{%
  \,\raisebox{-.127ex}{\reflectbox{\rotatebox[origin=br]{-90}{$\lnot$}}}\,}    

\newcommand{\noi}{\noindent}

\newcommand\blfootnote[1]{
  \begingroup
  \renewcommand\thefootnote{}\footnote{#1}%
  \addtocounter{footnote}{-1}%
  \endgroup
 }

\theoremstyle{definition} \newtheorem{de}{Definition}[section]
                          
\theoremstyle{remark}     
\theoremstyle{plain}      \newtheorem{thm}[de]{Theorem}
		          \newtheorem*{thm*}{Theorem}
\theoremstyle{plain}      \newtheorem{cor}[de]{Corollary}
\theoremstyle{plain}      \newtheorem{pr}[de]{Proposition}
\theoremstyle{plain}      \newtheorem{lem}[de]{Lemma}
\theoremstyle{remark}     \newtheorem{rmk}[de]{Remark}
\theoremstyle{remark}     
		          \newtheorem*{rmk*}{Remark}
		          
\theoremstyle{plain}

\author{Stefania Lisai and Mark Wilkinson}
\title {Smooth Solutions of the Surface Semi-Geostrophic Equations}
\date{}

\begin{document}

\maketitle


\begin{abstract}
\noindent The semi-geostrophic equations have attracted the attention of the physical and 
mathematical communities since the work of Hoskins in the 1970s owing to their ability 
to model the formation of fronts in rotation-dominated flows, and also to their connection with 
optimal transport theory.
In this paper, we study an active scalar equation, whose activity is determined by way of a Neumann-to-Dirichlet map associated to a fully nonlinear second-order Neumann boundary value problem
on the infinite strip $\R^2\times(0,1)$,
that models a semi-geostrophic flow in regime of constant potential vorticity. This system is an expression of an Eulerian semi-geostrophic flow in a co-ordinate system originally due to Hoskins, to which we shall refer as {\em Hoskins' coordinates}. We obtain results on the local-in-time existence and uniqueness of classical solutions of this active scalar equation in H\"older spaces.
\end{abstract}


\blfootnote{\textit{2010 Mathematics Subject Classification}. 35Q35, 76B03.\\
\textit{Key words and phrases}. active scalar equation; 
        surface semi-geostrophic equations; 
        Dirichlet-to-Neumann map; 
        Schauder theory; 
        Monge-Amp\`ere operator.\\
{\textsc{Maxwell Institute for Mathematical Sciences,
                    Department of Mathematics,
                    Heriot-Watt University,
                    Edinburgh,
                    UK
                    EH14 4AS. }
  \textit{Emails}: \href{mailto:s.lisai@sms.ed.ac.uk}{s.lisai@sms.ed.ac.uk}, 
                \href{mailto:mark.wilkinson@hw.ac.uk}{mark.wilkinson@hw.ac.uk}}}

\renewcommand{\thefootnote}{\arabic{footnote}}

\section{Introduction}

The semi-geostrophic equations (or SG for brevity) constitute a model for the large-scale dynamics of
atmospheres and oceans which are dominated by rotational effects. The equations 
take the form of an active semilinear transport equation in an unknown conservative vector 
field, and can be considered as a formal vanishing Rossby number limit of the well-known
primitive equations; see Section \ref{subsec:SG_coord} for the form of the semi-geostrophic equations in Eulerian coordinates. SG has attracted considerable
attention from the mathematical community over the past 20 years as their analysis
can be tackled using tools from optimal transport theory and the regularity theory 
of Alexandrov solutions of the Monge-Amp\`{e}re equation. In this paper, following
Hoskins \cite{Hoskins75}, we restrict our attention to incompressible semi-geostrophic
flows on an infinite strip $\O:=\R^{2}\times (0, 1)$ in the regime of constant potential vorticity. 
These are modelled by the following active scalar equation (to which we refer as SSG) {\em on the boundary} 
$\partial\Omega$ of the strip $\Omega$ in the unknown buoyancy anomaly $\theta$, namely
\begin{equation}\label{eq:SSG}
\left\{
\begin{array}{l}
\partial_{t}\theta+(w\cdot\nabla)\theta=0, \vspace{2mm}\\
w=\nabla^{\perp}\Top[\theta],
\end{array}
\right.
\end{equation}
where $\nabla^{\perp}$ is the $\pi/2$-rotated gradient operator on $\R^{2}$, 
and $\Top$ is the Neumann-to-Dirichlet map associated to the following 
time-independent second order fully nonlinear Neumann boundary value problem
given by
\begin{equation}\label{eq:BVP}
\left\{
\begin{array}{l}
F(D^{2}\Phi)=0 \quad \text{on}\hspace{2mm}\Omega, \vspace{2mm}\\
\partial_{n}\Phi=\theta \quad \text{on}\hspace{2mm}\partial\Omega,
\end{array}
\right.
\end{equation}
and $F:\mathbb{R}^{3\times 3}\rightarrow\mathbb{R}$ is defined 
pointwise by $F(A):=A_{11}+A_{22}+A_{33}-A_{11}A_{22}+A_{12}A_{21}$ 
for all $A\in\mathbb{R}^{3\times 3}$.
The reader will notice that the system \eqref{eq:SSG} is formally equivalent to the surface quasi-geostrophic model (SQG in short) except that the operator $\mathcal T$ is given by $(-\Delta)^{-1/2}$: see \cite{CMT94:SQG} for details.

In this paper, we construct local-in-time smooth solutions of system 
\ref{eq:SSG} by way of a double fixed point argument in spaces of H\"older continuous functions.
The analysis of system \eqref{eq:BVP} is tackled, in the regime of 
small boundary data, by means of classical 
elliptic theory.
Although our main existence result holds only for {\em local-in-time} smooth solutions of
\ref{eq:SSG}, as opposed to \emph{global-in-time} smooth solutions thereof, 
it is natural
to expect that the dynamics of the system produces discontinuous solutions
in finite time.
Indeed, in the original widely-cited work of Hoskins and Bretherton \cite{HB72}, 
the authors provide evidence of finite-time singularity generation of the
semi-geostrophic equations through a numerical study
on the infinite strip $\O$.

\subsection{Semi-geostrophic dynamics expressed in various coordinate systems} \label{subsec:SG_coord}

In this and the following section, we provide a brief overview of the 
state-of-the-art in results on the semi-geostrophic equations.
Indeed, the semi-geostrophic equations in 
Eulerian coordinates, derived in \cite{Eliassen48}, in the
regime of an incompressible and inviscid flow, comprises the 
following system
\begin{equation}\label{eq:SG1}
	\left\{
	\begin{array}{l}
	\dt u^g + (u\cdot \nabla) u^g = -Ju^a,\vspace{2mm}\\
	\dt \ta + (u\cdot\nabla) \ta =0,\vspace{2mm}\\
	\div u = 0,
	\end{array}
	\right.
\end{equation}
where $u^g:=(u^g_1, u^g_2, 0 )$ is the so-called 
\emph{geostrophic velocity} field which is that part of the Eulerian
velocity field $u$ in perfect geostrophic balance, and $u^a$ is 
the associated \emph{ageostrophic velocity} field, given by
\begin{equation*}
	u^a := u - u^g.
\end{equation*}
Moreover, the matrix $J\in\mathbb{R}^{3\times 3}$ is given by
\begin{equation*}
	J = 
	\begin{pmatrix}
        0 & -1 & 0\\
        1 & 0 & 0\\
        0 & 0 & 0
	\end{pmatrix}.
\end{equation*}
The geostrophic velocity $u^g$ and buoyancy anomaly $\theta$ are
not independent quantities, but rather are realised as the 
gradient of a scalar pressure $\phi$, namely
\begin{equation}\notag
	\nabla \phi = \begin{pmatrix}
                u^g_2\vspace{2mm} \\ -u^g_1\vspace{2mm} 
                \\ {\ta} 
	           \end{pmatrix}.
\end{equation}
We refer to \eqref{eq:SG1} simply as SG in all that follows.

The construction of any notion of solution (either classical or 
distributional) of SG expressed in Eulerian coordinates is a difficult problem,
and only few results in this direction exist in the literature.
By drawing a brief analogy with the theory of water waves, or the free-surface Euler equations 
(see \cite{lannes:waterwaves}),
in the study of SG it is useful to rewrite the governing equations in 
different coordinate systems,
with the hope of constructing some solution of the formally 
equivalent system therein. 
Let us now  present for the convenience of the reader those versions
of SG that have been studied to date
and discuss their relationships briefly.
Moreover, we provide a mathematical derivation of \eqref{eq:SSG} and \eqref{eq:BVP}, 
following Hoskins in \cite{Hoskins75}, in Section \ref{sec:deriv_SSG}.

By introducing the scalar field $P$ defined on $\Omega\times \R$ as
\begin{equation}\notag
	P (x,t) := \phi(x,t) + \frac 12 (x_1^2 + x_2^2),
\end{equation}
often called the \emph{generalised pressure} or \emph{generalised geopotential},
system \eqref{eq:SG1} is equivalent to a semilinear transport equation in the unknown conservative vector fields $\nabla P$ and $u$,
\begin{equation}\label{eq:SG2}
	\left\{
	\begin{array}{l}
		\dt\nabla P + D^2 P u = J (\nabla P- \id_{\O}),\vspace{2mm}\\
		\div u = 0.
	\end{array}
	\right.
\end{equation}
Here, $\id_{\O}$ denotes the identity map on $\O$, $\id_{\O}:x\mapsto x$.
Notably, the system \eqref{eq:SG2} is not supplemented with an evolution equation for the velocity
vector field $u$; rather, the velocity field $u$ must evolve in such a way that the time-dependent vector field $\nabla P$ remain conservative. One formally equivalent formulation of SG considered in the literature is that in so-called 
\emph{Lagrangian coordinates}, namely 
\begin{equation}\label{eq:SG3}
	\left\{
	\begin{array}{l}
		\dt T = J(T-X),\vspace{2mm}\\
		\dt X(x,t) = u(X(x,t),t),\vspace{2mm}\\
		\div u =0,
	\end{array}
	\right.
\end{equation}
where $X$ is the Lagrangian flow corresponding to $u$, and the relation
between the unknown $T$ and $P$ is given by
\begin{equation}\label{eq:SG_Lagflow}
T(x,t):=\nabla P(X(x,t),t).
\end{equation}
Evidently, the equivalence between the Lagrangian formulation and the previous Eulerian one depends on the regularity of the velocity vector field $u$. One other formally equivalent formulation is expressed in so-called \emph{geostrophic coordinates}, as studied for instance in the well-known work of Benamou and Brenier \cite{BB98}. Indeed, if we assume that the map $\nabla P(\cdot,t):\O\to \R^3$ 
is invertible for any $t>0$, then the scalar field $\al$ defined as 
\begin{equation}\label{eq:SG_MA}
\al(\cdot,t):=\det D^2 P^* (\cdot,t)
\end{equation}
satisfies a transport equation with a vector field that depends on $\nabla P^*$, namely
\begin{equation}\label{eq:SG4}
	\left\{
	\begin{array}{l}
		\dt \al + (W\cdot\nabla)\al = 0,\vspace{2mm}\\
		W(\cdot,t):= J(\id_{\R^3}-\nabla P^*(\cdot,t)).
	\end{array}
	\right.
\end{equation}
This system is also an example of an active scalar equation, where the activity is determined by way of a weak solution of the second boundary-value problem for the Monge-Amp\`{e}re equation. As is standard, we use the notation $P^*$ to denote the Legendre transform of $P$,
while $\id_{\R^3}:x\mapsto x$ denotes the identity map on $\R^3$. We now discuss the results in the literature regarding the well-posedness 
of the previous formulations of the system \eqref{eq:SG1}.

\subsection{Brief Review of Existence Results for SG}

In \cite{BB98}, Benamou and Brenier provided the first result on existence of weak solutions to SG in geostrophic co-ordinates \eqref{eq:SG4}
in full geostrophic coordinates. The main theorem proved in \cite{BB98} is the following:
\begin{thm*}[\cite{BB98}]
	Let $\O\subset \R^3$ be a bounded Lipschitz open set and 
	$\al_0\in L^p(\R^3)$ be of compact support.
	For any $\tau>0$ and $p>3$, there exist
	\begin{itemize}
		\item [(i)] $R>0$ with $B:=B(0,R)\subset\R^3$;
		\item [(ii)] $\al \in\linf (0,\tau;L^p(B))$;
		\item [(iii)] $P \in \linf (0,\tau;W^{1,\infty}(\O))$ with $P(\cdot, t)$ convex;
		\item [(iv)] $P^{\ast} \in \linf (0,\tau;W^{1,\infty}(\R^3) )$; and
		\item [(v)] $W \in \linf (0,\tau;\linf_{\loc}(\R^3)\cap BV_{\loc}(\R^3))$
	\end{itemize}
such that $(\alpha, W)$ furnishes a distributional solution of \eqref{eq:SG_MA} and \eqref{eq:SG4}, with $\al(\cdot,0)=\al_0$ in $L^p(\R^3)$.
\end{thm*}
The sense in which the quantities above solve the Cauchy problem is 
defined in Section 5 of \cite{BB98}.
The authors perform a time-stepping argument, and 
solve the Monge-Amp\`ere equation at each time step using the fundamental
results of \cite{Brenier91}. Their result is compatible with 
the independently-derived stability principle of Cullen and Shutts in \cite{CS87}:
in this work, the authors affirm that the solutions of SG that are \emph{stable} (in the sense of inner variation) are those 
for which the generalised pressure $P$ is convex. This result is the first clear connection
between the semi-geostrophic equations and the theory of optimal transport.
In \cite{CG01}, Cullen and Gangbo extend the result by Benamou and Brenier to 
the shallow water regime, which requires constant potential temperature $\ta$
on a free-surface. In \cite{CM03}, Cullen and Maroofi further extend the results from 
\cite{BB98} and \cite{CG01} to the case of fully compressible semi-geostrophic flow. These results deal with SG in full geostrophic 
coordinates, as the regularity of the obtained solutions is not sufficient to construct solutions in physical coordinates, either Eulerian or
Lagrangian. 

In \cite{CF06}, Cullen and Feldman make use of the 
theory by Ambrosio \cite{Ambrosio04:ODE} on the transport equation and ODEs
with vector fields of bounded variation, to prove existence of weak solutions to SG in Lagrangian 
coordinates \eqref{eq:SG3}, both in domains in $\mathbb{R}^{3}$ and in the regime of shallow water.
The main result by Cullen and Feldman is the following.
\begin{thm*}[\cite{CF06}]
	Let $\O\subset\R^3$ be an open bounded subset. 
	Let $P_0$ be a bounded convex function on $\Omega$ and assume that 
	$\nabla P_0 \# \mathcal L^3\mres \O \ll\mathcal L^3$ with density in 
	$L^q(\R^3)$, for some $1<q<\infty$. Then for any $\tau>0$ there exist
	\begin{itemize}
        \item [(i)] $r\in[1,\infty)$;
		\item [(ii)] $P\in\linf\left(0, \tau; W^{1,\infty}(\O)\right)\cap C([0,T);W^{1,r}(\O))$
            with $P(t,\cdot)$ convex for any $t\in[0,T)$; and
        \item [(iii)] $X: [0, \tau)\rightarrow L^r(\O;\R^3)$ Borel map
	\end{itemize}
	such that $(P,X)$ is a weak Lagrangian solution of SG in 
	Lagrangian coordinates \eqref{eq:SG3} and \eqref{eq:SG_Lagflow}
	on $\O\times [0, T)$.
\end{thm*}
We invite the reader to consult definition 2.5 in \cite{CF06} for   
the sense in which $(P,X)$ is a 
weak solution of \eqref{eq:SG3} and \eqref{eq:SG_Lagflow}.
It is also worth mentioning the work by Faria, Lopes Filho and Nussenzveig-Lopes 
\cite{Faria09} in which the authors extend the result in \cite{CF06} to the borderline Lebesgue index case $q=1$, i.e. the case in which $\nabla P_0\# \mathcal L^3\mres \O$ has density in $L^1(\R^3)$.
It is still not clear if one might use these solutions 
to construct a solution in Eulerian coordinates. Uniqueness of these weak solutions also remains an open question.

The problem of uniqueness of solutions in any context was addressed by Loeper in \cite{Loeper06}, in which the 
author proves existence and stability (in the sense of Shutts and Cullen)
of measure-valued solutions to SG in full geostrophic coordinates \eqref{eq:SG4}
on the torus $\T^3$. In this work, he also studies regularity and uniqueness of smooth solutions, and 
explores analytical similarities between SG \eqref{eq:SG4} and the well-known
2-D incompressible Euler equations in vorticity formulation, namely
\begin{equation*}
	\left\{
	\begin{array}{l}
		\dt \o + \nabla \cdot (\o v)=0,\vspace{1mm}\\
		v = \nabla^{\perp}\Phi,\vspace{1mm}\\
		\Dl \Phi = \o.
	\end{array}
    \right.
\end{equation*}
The smooth solutions constructed on $\T^3$ by Loeper admit the property that $\nabla P(\cdot,t)$ is a diffeomorphism, whence 
smooth solutions in full geostrophic coordinates of \eqref{eq:SG4} 
can be used to construct classical solutions in Eulerian coordinates of 
\eqref{eq:SG2}. This is the first result on existence of Eulerian solutions,
however the condition of having 
$\supp (\nabla P(\cdot,t)\#\mathcal L^3 \mres \O) = \R^3$ implies that
the vector field $\nabla P(\cdot,t)$ cannot lie in $\linf(\O)$ for any time $t$. 
This poses an issue if one is interested in the physical application of the model SG, 
as such solutions correspond to an unbounded potential temperature, which is given by $\dz P(\cdot,t)=\ta(\cdot,t)$.

Following the result of De Philippis and Figalli \cite{DF12:reg_MA} on higher regularity of Alexandrov solutions
to the Monge-Amp\`ere equation, it became possible 
to improve the regularity of weak solutions $\nabla P(\cdot,t)$ constructed by Benamou and Brenier to the class 
$W^{1,1}$. Indeed, in \cite{ACDF12:convex} and \cite{ACDF12:periodic}, Ambrosio, Colombo,
De Philippis and Figalli prove the existence of global weak solutions of
SG \eqref{eq:SG2} in Eulerian coordinates in the case of a convex 3-D domain $\O$
in Eulerian coordinates
and on the 2-D torus $\T^2$ respectively,
in the case in which $\nabla P_t\#\mathcal L^3\mres \O$ is assumed bounded above and 
away from zero. 

In \cite{FT13}, Feldman and Tudorascu demonstrate that weak solutions
in Eulerian coordinates have the property that the measure $\nabla P(\cdot,t)\#\mathcal L^3\mres\O$
has no atomic part. In order to allow such a case, which is physically pertinent as it corresponds to particular frontal singularities, the authors define a notion of generalised weak solution of SG in Lagrangian coordinates and prove the existence thereof. The authors improve upon this result
in \cite{FT15}, in which they prove existence of \emph{relaxed} Lagrangian
solutions to SG on a domain in $\mathbb{R}^{3}$
with any general initial data $P_0 \in H^1(\O)$ which is convex.
Finally, in \cite{FT17}, Feldman and Tudorascu address the problem of uniqueness of solutions of SG. In particular, they demonstrate weak-strong uniqueness in Lagrangian coordinates under the assumptions of boundedness of the Eulerian velocity field $u$ and uniform convexity of $P^*$.
In \cite{Wilkinson18:SG}, the second author proves existence of local-in-time classical solutions of \eqref{eq:SG1} in Eulerian coordinates on 3-dimensional smooth bounded simply-connected domains. The technique used relies on the theory of the so-called div-curl systems, and it is consistent with the Stability Principle of Cullen and Shutts.

 
\subsection{Main Result} 

In this paper, we prove the existence and uniqueness 
of classical solutions of the initial value problem associated to \eqref{eq:SSG} for given smooth initial
data. As a minor simplification of the full model originally introduced in Hoskins \cite{Hoskins75}, we restrict our interest to the case in which $\ta=0$ on 
the lower boundary $\R^2\times\{0\}\subset\partial\Omega$, so that one need only deal with a single evolution equation 
\eqref{eq:SSG} on the upper boundary $\R^2\times\{1\}$, as opposed to \emph{two} coupled equations on the disconnected set $\partial\Omega$. We comment further on this point in the derivation of the surface semi-geostrophic model from the full semi-geostrophic equations in Section \ref{sec:deriv_SSG} below. We also restrict ourselves to considering SSG on the 
flat torus $\T^2$, instead of the unbounded plane $\R^2$.

In all that follows, we say that a pair of maps $(\ta,w)$ is a 
\emph{local-in-time classical solution} of \eqref{eq:SSG} associated to initial datum $\ta_0\in C^1(\T^2)$
if there exists $\tau>0$ such that $\ta\in C^1((0,\tau);C^1(\T^2))$,
$w\in C^0([0,\tau];C^1_{\s}(\T^2;\R^2))$, with $\ta$ and $w$ solving \eqref{eq:SSG} pointwise
everywhere in $\T^2\times (0,\tau)$, together with $\ta(\cdot,0)=\ta_0$. The main result of this paper is the following theorem. We use here the notation $\mathcal C^1_{\s}(\T^2;\R^2)$ to indicate the space of divergence-free functions in $\mathcal C^1(\T^2;\R^2)$.

\begin{thm}[Local-in-time existence and uniqueness of classical solutions to SSG]
\label{thm:SSG}
Suppose $k\in\N$. There exists $\rho_{k}>0$ such that given $\ta_0\in\cal {k+2}(\T^2)$ with $\|\ta_0\|_{\cal {k+2}}\leq \rho_{k}$
 and $\int_{\T^2}\ta_0=0$,
 there exists a $\tau_{k}=\tau_{k}(\ta_0)>0$ such that there is an associated 
 local-in-time classical solution $(\ta,w)$ of \eqref{eq:SSG} on $\T^2\times (0,\tau_k)$.
Moreover, for $k\geq 2$ the classical solution is unique.
\end{thm}

As we have presented above, the active vector field $w$ depends on the unknown $\ta$ through the Neumann-to-Dirichlet operator associated to the Neumann BVP for the fully-nonlinear equation \eqref{eq:BVP}. Therefore, an important part of our proof consists in proving that such an operator is well-defined and admits useful analytical properties.
\begin{rmk}
Whilst we believe it is possible to establish the analogue of Theorem \ref{thm:SSG} in the case that $\theta_{0}$ is smooth and non-periodic on $\mathbb{R}^{2}$, we do not do this here.
\end{rmk}

\subsection{Structure of Paper}
The paper is structured as follows.
In Section \ref{sec:deriv_SSG}, we present the derivation of SSG from SG, as originally performed in 
\cite{Hoskins75}. 
In Section \ref{sec:NtD}, we introduce a construction of the 
Neumann-to-Dirichlet operator $\Top$ defined on H\"older spaces, and discuss some of its relevant properties.
In Section \ref{sec:SSG}, we prove Theorem \ref{thm:SSG} through an application of Schauder's
fixed point theorem, making use of classical estimates on the solutions of passive transport equations on $\mathbb{T}^{2}$.
In the closing Section \ref{sec:concl}, we discuss some natural generalisations of our result. Finally, in the Appendix \ref{sec:app}, for the reader's convenience, we provide details of the calculations underlying the arguments
in the previous sections.


\section{Derivation of SSG from SG}\label{sec:deriv_SSG}

The surface semi-geostrophic equations were derived 
from the semi-geostrophic equations by Hoskins \cite{Hoskins75} in 1975. They arise when one considers the special case of solutions of SG which admit spatially-homogeneous {\em potential vorticity}. For the convenience of the reader, we reproduce a derivation of this model here, starting from classical solutions of SG in Eulerian coordinates (which are assumed, but are not known, to exist). It will be helpful in the sequel to consider the system \eqref{eq:SG1} expressed in all its components, namely
\begin{equation}\label{eq:SG_Eul}
\left\{
    \begin{array}{l}
		{\displaystyle (\dt + u\cdot \nabla) u^g_1 - f u_2 + \pdx {\phi} = 0},  \vspace{2mm}\\
		{\displaystyle (\dt + u\cdot \nabla) u^g_2 + f u_1 + \pdy {\phi} = 0} , \vspace{2mm}\\
		(\dt + u\cdot\nabla)\ta = 0,  \vspace{2mm}\\
		{\displaystyle \frac{g}{\ta_0} \ta = \pdz {\phi}},  \vspace{2mm}\\

		{\displaystyle u^g_1 = -\frac 1f \pdy {\phi}},  \vspace{2mm}\\
		{\displaystyle u^g_2 = \frac 1f \pdx {\phi}}, \vspace{2mm} \\
		\div u = 0,
	\end{array}
\right.
\end{equation}
where $f>0$ is the Coriolis parameter, assumed constant in what follows. Moreover, the Eulerian velocity field $u$ is subject to the no-flux constraint $u(\cdot, t)\cdot n=0$ on $\partial\O$ for all times $t$, where $n:\partial\Omega\rightarrow\mathbb{S}^{2}$ denotes the outward unit normal map. As we shall see below, this no-flux assumption is crucial in the derivation of the surface semi-geostrophic equations. We define the {\em vorticity} $\zeta_{g}$ associated to the dynamics of system \eqref{eq:SG_Eul} by
\begin{align}\notag
	\zeta_g &:= \left( -\pdz {u^g_2},\ \pdz {u^g_1},\ f + \pdx {u^g_2} - \pdy {u^g_1} \right)  
                + \frac 1f \left( \parder{(u^g_1, u^g_2)}{(x_1,x_2)},\ 
                                \parder{(u^g_1, u^g_2)}{(x_1,x_3)},\ 
                                \parder{(u^g_1, u^g_2)}{(x_2,x_3)} \right), \label{eq:def_vorticity}
\end{align}
and one may readily check that it satisfies the following {\em vorticity equation}, namely
\begin{equation*}
	(\dt + u\cdot\nabla)\zeta_g = (\zeta_g\cdot \nabla)u - \frac g{\ta_0} e_3\wedge\nabla\ta
\end{equation*}
pointwise in the classical sense on space-time. If one in turn defines the \emph{potential vorticity} $q_{g}$ of the geostrophic flow as 
\begin{equation}\label{eq:def_PV}
	q_g := \zeta_g\cdot\nabla\ta,
\end{equation}
it follows from a calculation that the potential vorticity is conserved along Lagrangian particle trajectories, namely
\begin{equation}\notag
	(\dt + u\cdot \nabla)q_g = 0.
\end{equation}
As such, if one furnishes the system \eqref{eq:SG_Eul} with initial data $(u^{g}_{1, 0}, u^{g}_{2, 0}, \theta_{0})$ whose associated potential vorticity is spatially inhomogenous, then the corresponding solution $(u^{g}_{1}, u^{g}_{2}, \theta, u)$ also formally has this property. It was discovered by Hoskins that solutions which admit constant potential vorticity on $\Omega$ admit a rather beautiful structure in another coordinate system (to which we henceforth refer as {\em Hoskins' coordinates}), details of which we now provide.

\subsection{Transformation of Coordinates}
Suppose smooth initial data $(u^{g}_{1, 0}, u^{g}_{2, 0}, \theta_{0})$ for the dynamics formally generated by \eqref{eq:SG_Eul} are given. We now consider the associated 1-parameter family $\{H_t\}_{t\geq 0}$ of smooth maps
$H_t:\O\to\O$ defined by
\begin{equation}\notag
	H_t(x):=\begin{pmatrix}
	        	x_1 + \frac 1f u^g_2(x,t) \vspace{1mm} \\
	        	x_2 - \frac 1f u^g_1(x,t) \vspace{1mm} \\
	        	x_3
	        \end{pmatrix} \quad \text{for}\hspace{2mm}x\in\Omega.
\end{equation}
For each time $t$, the map $H_t$ is assumed to be 
a $C^1$-diffeomorphism in what follows. In the sequel, we shall use capital Roman letters, namely $X=(X_1,X_2,X_3)$, to denote the independent variable for maps defined on $\Omega$ considered as the range space of the coordinate transformation $H_{t}$ for any $t$. Our aim in the sequel is to close a system of equations for a number of `natural' quantities in the coordinate system determined by $\{H_{t}\}_{t\geq 0}$. In this pursuit, we begin with the following proposition.
\begin{pr}\label{manip}
For any smooth map $\Psi:\Omega\times \mathbb{R}\rightarrow\mathbb{R}$, it holds that
\begin{equation*}
        (\dt + u(x,t)\cdot\nabla_x)\Psi(H_t(x),t)= 
            \left(\dt + U(H_t(x),t)\cdot\nabla_X\right)\Psi (H_t(x),t),
\end{equation*}
where $U(X,t):= (u^g_1(H_t^{-1}(X),t), u^g_2(H_t^{-1}(X),t), u_3(H_t^{-1}(X),t))$ for $X\in\Omega$ and each time $t$. 
\end{pr}
\begin{proof}
We begin by noticing that for any smooth map $\Psi=\Psi(X,t)$, the material derivative of the composition $(x, t)\mapsto \Psi(H_{t}(x), t)$ (with respect to the Eulerian velocity field $u$) in Eulerian coordinates is given by 
\begin{equation*}
        (\dt + u(x,t)\cdot\nabla_x)\Psi(H_t(x),t) = 
            \dt \Psi (H_t(x),t) + (\partial_{t}H(x, t) + (D_{x}H(x, t))^{T}u(x,t))\cdot\nabla_X\Psi (H_t(x),t).
\end{equation*}
As such, the corresponding velocity field that advects the flow in Hoskins' coordinate system is given simply by $\partial_{t}H+(D_{x}H)^{T}u$. If we consider its components, namely
    \begin{align*}
        (\dt + u(x,t)\cdot \nabla_x)H_1 (x,t) 
            &= u_1(x,t) + \frac 1f \left( - fu_1(x,t) - \pdy{\phi}(x,t) \right)\\
            &= -\frac 1f \pdy{\phi}(x,t)
                = u^g_1(x,t),\\
        (\dt + u(x,t)\cdot \nabla_x)H_2 (x,t) 
            &= u_2(x,t) - \frac 1f \left( fu_2(x,t) - \pdx{\phi}(x,t) \right)\\
            &= \frac 1f \pdx{\phi}(x,t)
                = u^g_2(x,t),\\
        (\dt + u(x,t)\cdot \nabla_x)H_3 (x,t)&=u_3(x,t),
    \end{align*}
it follows simply that
\begin{equation}\notag
            \partial_{t}H+(D_{x}H)^{T}u = (u^g_1, u^g_2, u_3),
\end{equation}
whence follows the proof of the claim.
\end{proof}
It will be helpful in what follows to note that the matrices $D_xH_t$ and $D_X H_t^{-1}$ are explicitly given by the following expressions:
\begin{equation}\notag
	D_x H_t(x,t) = \begin{pmatrix}
	              	1+ \frac 1{f^2} {\phi}_{11} & \frac 1{f^2} {\phi}_{12} & 0 \\
	              	\frac 1{f^2} {\phi}_{12} & 1+ \frac 1{f^2} {\phi}_{22} & 0 \\
	              	\frac 1{f^2} {\phi}_{13} & \frac 1{f^2} {\phi}_{23} & 1
	              \end{pmatrix}(x,t),
\end{equation}
and
\begin{align}\notag
	& j(x, t)D_X H_t^{-1} (H_t(x))  \\ \label{eq:DX_Ht_inv}
        &\quad=\begin{pmatrix}
        	1+ \frac 1{f^2} \phi_{22} & - \frac 1{f^2} \phi_{12} & 0 \\
        	- \frac 1{f^2} \phi_{12} & 1+ \frac 1{f^2} \phi_{11} & 0 \\
        	- \frac 1{f^2}\phi_{13} + \frac 1{f^4}\bigg(\phi_{12}\phi_{23} - \phi_{13}\phi_{22}\bigg)
          & - \frac 1{f^2}\phi_{23} + \frac 1{f^4}\bigg(\phi_{12}\phi_{13} - \phi_{11}\phi_{23}\bigg)
          & j
        \end{pmatrix}(x ,t),
\end{align}
where we used the notation $\phi_{ij}=\frac{\partial^2\phi}{\partial x_i\partial x_j}$ and $j$ denotes the Jacobian of $H_t$, namely
\begin{equation*}
	j(x,t) := \det D_xH_t(x) 
	= \left(1 + \frac 1{f^2}\pdxx{\phi}(x,t)\right)
	\left(1+ \frac 1{f^2}\pdyy{\phi}(x,t)\right) 
	- \frac 1{f^4}\left(\pdxy{\phi}(x,t)\right)^2.
\end{equation*}
We shall also employ the notation $J(X,t):=j(H_t^{-1}(X),t)$ in what follows. The following observation holds as a consequence of Proposition \ref{manip} above.
\begin{cor}
Suppose $(u^{g}_{1}, u^{g}_{2}, \theta, u)$ is a smooth solution of \eqref{eq:SG_Eul}. It follows that both the maps $\Theta$ and $Q$ defined pointwise as
\begin{equation*}
\Theta(X, t):=\theta(H_{t}^{-1}(X), t) \quad \text{and}\quad Q(X, t):=q_{g}(H_{t}^{-1}(X), t)
\end{equation*}
are smooth solutions of the transport equation
\begin{equation*}
(\partial_{t}+U\cdot\nabla)\Psi=0.
\end{equation*}
\end{cor}
\begin{rmk}
In all that follows, we refer to $\Theta$ and $Q$ as the geostrophic buoyancy anomaly and geostrophic potential vorticity, respectively.
\end{rmk}
\begin{proof}
By the definition of vorticity $\zeta_{g}$ and the form of the matrix $D_X H_t^{-1}$ in \eqref{eq:DX_Ht_inv}, one observes that the third row of $D_X H_t^{-1}$ is given simply by a multiple of the vorticity $\zeta_g$, namely
\begin{equation*}
\frac 1f \zeta_g(x,t) = j(x,t)(D_X H_t^{-1}(H_t(x)))_3.
\end{equation*}
For any smooth map $\psi\equiv\psi(x)$, it follows that
\begin{align*}
\pdZ{} \left(\psi(H_t^{-1}(X))\right) &= \frac 1{f J(X,t)} \zeta_g(H_t^{-1}(X),t)\cdot \nabla_x \psi(H_t^{-1}(X)),
\end{align*}
with, in particular, it being the case that
\begin{align}\notag
	\pdZ{\Ta}(X,t)=\frac 1{fJ(X,t)} Q(X,t),
\end{align}
by the definition of potential vorticity \eqref{eq:def_PV}.
Since the potential temperature $\ta$ and the 
potential vorticity $q_g$ are Lagrangian invariants with respect to the Eulerian flow $u$, it follows that $\Theta$ and $Q$ are Lagrangian invariants with respect to the flow $U$, namely
 \begin{align}\label{eq:conserv_Ta}
	(\dt + U\cdot \nabla_X)\Ta = 0,\vspace{1mm}\\
	\label{eq:conserv_Q}
	(\dt + U\cdot\nabla_X)Q = 0,
\end{align}
which concludes the proof.
\end{proof}
The study of $Q$ and $\Theta$ in the coordinate system determined by $\{H_{t}\}_{t\geq 0}$ will, in some sense, `replace' the study of $u$ and $\nabla P$ in Eulerian coordinates. To see how this is the case, we introduce an important streamfunction originally due to Hoskins.

\subsection{A New Streamfunction in Geostrophic Coordinates}
An important observation of Hoskins was that there exists a time-dependent potential $\Phi$ which `stores' all salient features of the dynamics in geostrophic coordinates. Indeed, following \cite[Section~4]{Hoskins75}, we define a streamfunction $\Phi$ in Hoskins' coordinates $\Phi:\O\times (0, \infty)\rightarrow\mathbb{R}$ by
\begin{equation*}
	\Phi(X,t) : = \phi (H_t^{-1}(X),t)
        + \frac 12 \left( u^g_1(H_t^{-1}(X),t)^2
                        + u^g_2(H_t^{-1}(X),t)^2 \right)
\end{equation*}
for $X\in\Omega$ and any $t>0$. A straightforward calculation reveals that
\begin{align*}
	D_x H_t(x)\nabla_X \Phi(H_t(x),t) 
        &= \nabla_x \phi (x,t)
                + u^g_1(x,t) \nabla_x u^g_1(x,t)
                + u^g_2(x,t) \nabla_x u^g_2(x,t)\\
        &=\nabla_x\phi(x,t) 
                + \frac 1{f^2}\pdx {\phi}(x,t)\nabla_x\left(\pdx{\phi}(x,t)\right)
                + \frac 1{f^2}\pdy {\phi}(x,t)\nabla_x\left(\pdy{\phi}(x,t)\right)\\
        &=\nabla_x\phi(x,t) + \begin{pmatrix}
                              	\pdxx{\phi}(x,t) & \pdxy{\phi}(x,t) & 0\\
                              	\pdxy{\phi}(x,t) & \pdyy{\phi}(x,t) & 0\\
                              	\pdxz{\phi}(x,t) & \pdyz{\phi}(x,t) & 0
                              \end{pmatrix} \nabla_x \phi(x,t)\\
        &= D_x H_t(x) \nabla_x\phi(x,t).
\end{align*}
Therefore, by the assumption that $H_{t}$ be a diffeomorphism of $\Omega$ for all times $t$, we have that the following identity between the gradients of the 
streamfunction $\Phi$ and the pressure $\phi$ holds true:
\begin{equation*}
\nabla_X\Phi(X,t) = \nabla_x\phi(H_t^{-1}(X),t).
\end{equation*}
It is this identity which furnishes the link between the semi-geostrophic equations in Eulerian coordinates and the so-called surface semi-geostrophic equations in Hoskins' coordinates. In order to understand the dynamics of $\Phi$ in geostrophic
coordinates, we observe the following identities which relate the second derivatives of the pressure $\phi$ to those of the streamfunction $\Phi$:
\begin{align*}
	1+ \frac 1{f^2} \pdxx\phi(H_t^{-1}(X),t) 
            &= J(X,t) \pdY{}(H_t^{-1}(X))_2\\
            &= J(X,t) \left(1-\frac 1{f^2}\pdYY{\Phi}(X,t)\right),\\
    1+ \frac 1{f^2} \pdyy\phi(H_t^{-1}(X),t)
            &= J(X,t)\pdX{} (H_t^{-1}(X))_1\\
            &= J(X,t) \left(1 - \frac 1{f^2} \pdXX \Phi (X,t) \right),\\
    \frac 1{f^2} \pdxy \phi(H_t^{-1}(X),t) 
            &= -J(X,t)\pdX(H_t^{-1}(X))_2\\
            &= -J(X,t)\frac 1{f^2} \pdXY\Phi(X,t).
\end{align*}
By taking products and differences in the above, we may identify an equation for the determinant $J$, namely
\begin{align*}
	J  = J ^2 \bigg\{\left(1-\frac 1{f^2}\pdXX{\Phi} \right)
            \left(1 - \frac 1{f^2} \pdYY \Phi   \right) 
            - \frac 1{f^4}\left( \pdXY\Phi \right)^2 \bigg\}.
\end{align*}
One can rewrite the determinant in the 
equation above in terms of the geostrophic potential vorticity $Q$ and the derivative $\dZ{\Ta}=\dZZ\Phi$ to obtain
\begin{equation}\label{eq:SSGderiv5}
	1 = \frac 1{f^2}\left( \pdXX \Phi + \pdYY\Phi\right) 
            + \frac{f\ta_0}{gQ}\pdZZ\Phi
            - \frac 1{f^4}\left(\pdXX\Phi\pdYY\Phi - \bigg(\pdXY\Phi\bigg)^2\right).
\end{equation}
It is the conservation laws \eqref{eq:conserv_Ta} and \eqref{eq:conserv_Q}
and the equation \eqref{eq:SSGderiv5} which comprise Hoskins' formulation of the semi-geostrophic equations \eqref{eq:SG_Eul} in the coordinate system determined by $\{H_{t}\}_{t\geq 0}$, namely
\begin{equation}\label{eq:SG_geo}
\left\{
    \begin{array}{l}
	{\displaystyle (\dt + U\cdot\nabla_X)\Ta = 0}, \vspace{2mm}\\
    {\displaystyle (\dt + U\cdot\nabla_X)Q =0}, \vspace{2mm}\\
    {\displaystyle 1 = \frac 1{f^2}\left( \pdXX \Phi 
        + \pdYY\Phi\right) + \frac{f\ta_0}{gQ}\pdZZ\Phi
        - \frac 1{f^4}\left(\pdXX\Phi\pdYY\Phi 
        - \bigg(\pdXY\Phi\bigg)^2\right)},
    \end{array}
\right.
\end{equation}
where the geostrophic velocity field $U$ is given by 
\begin{equation*}
U=\left(-\frac 1f\pdY\Phi, \frac 1f \pdX\Phi, u_3 \circ H_t^{-1}\right).
\end{equation*} 
From a structural point of view, system \eqref{eq:SG_geo} constitutes a fully nonlinear PDE (possibly of elliptic type) with non-constant coefficients which is coupled to a pair of transport equations associated to the geostrophic velocity field $U$. Moreover, the boundary condition $u_3=0$ on both components of $\dd \O$ is equivalent to $U_3=0$ on $\dd\O$ in Hoskins' geostrophic
coordinates. At this moment, we do not know that this system is {\em closed}, in the sense that the IBVP associated to \eqref{eq:SG_geo} admits a unique solution in any sense. It is at this point it is prudent to employ the additional assumption that initial data $(u^{g}_{1, 0}, u^{g}_{2, 0}, \theta_{0})$ are taken such that the value of $q_{g, 0}$ given by is constant on $\Omega$. This allows us to eliminate the transport equation for the geostrophic potential vorticity $Q$ and ensure that the fully nonlinear equation for $\Phi$ in \eqref{eq:SG_geo} admit constant coefficients. We outline the structure of the corresponding system below.

\subsection{Surface Semi-geostrophic Flow}

As intimated above, the equations \eqref{eq:SG_geo} that we obtain in Hoskins' geostrophic
coordinates are not easily studied. We make the additional important assumption of
constant potential vorticity at time $t=0$, namely
\begin{align*}
	q_{g, 0} = \frac{f\theta_0}g N^2 \qquad \Longleftrightarrow \qquad Q_{0}:=q_{g, 0}\circ H_{0}=\frac{f\theta_0}g N^2,
\end{align*}
where the constant $N$ is the Brunt–V\"ais\"al\"a frequency of the fluid.
Although this assumption might appear to be unnecessarily restrictive at first glance, it is in fact a good approximation when studying real-world atmospheric flows (see \cite{Juckes94:QG}). Mathematically on the other hand, with this simplification and subject to the rescaling
\begin{equation*}
 \widetilde{\Phi}(X_1, X_2, X_3)=\Phi\left(\frac{X_1}f, \frac{X_2}f, \frac{X_3}N\right) - \frac{X_3^2}2,
\end{equation*}
system \eqref{eq:SG_geo} reduces to the following set of coupled equations:
\begin{equation*}
\left\{
    \begin{array}{l}
	{\displaystyle \left(\dt -\frac 1f\pdY\Phi\pdX{} 
            + \frac 1f \pdX\Phi \pdY{} 
            + U_3 \pdZ{} \right)\pdZ \Phi = 0},  \vspace{1mm} \\
    {\displaystyle \Delta \Phi= \pdXX\Phi\pdYY\Phi 
        - \bigg(\pdXY\Phi\bigg)^2,} 
    \end{array}
\right.
\end{equation*}
together with the boundary condition $U_3=0$ on $\dd\O$, where the rescaled value of $\Phi$ has been simply relabelled as $\Phi$. It turns out that in this setting of spatially inhomogeneous potential vorticity, one need only solve the transport equation on each of the boundary planes $\partial\Omega$, as opposed to on the entire interior of the domain $\Omega$. Indeed, the nomenclature `surface semi-geostrophic equations' comes from the fact that the dynamics in geostrophic coordinates in the bulk is `determined' by the dynamics on the upper and lower boundary planes. Of course, in our work, as the dynamics on the lower boundary portion is trivial, it suffices to understand the evolution of the geostrophic buoyancy on the upper boundary plane alone. 


\section{Construction and Properties of the Neumann-to-Dirichlet Operator}\label{sec:NtD}

Our first step in the proof of our main result, namely Theorem \ref{thm:SSG}, is to demonstrate that the Neumann-to-Dirichlet map (whose formal definition is given in \ref{def:berp} below) associated to the fully nonlinear equation is well defined and admits useful analytical properties. Indeed, this pursuit is captured by the following theorem:

\begin{thm}\label{thm:nonlin_BVP}
    There exists a Fr\'echet differentiable map $\mathcal S$ defined on a ball $Y\subset\cal {k+1}(\T^2)$
    such that, for any $\ta\in Y$, $\mathcal S[\ta]$ is a $\cal {k+2}$-classical solution of 
    \begin{equation}\label{eq:BVP_zeromean}
    \left\{
    \begin{array}{l}
    \Dl \Phi =  \dxx\Phi\dyy\Phi - (\dxy\Phi)^2,  \vspace{1mm}\\
    \dz \Phi|_{\Gamma_0} =0,  \vspace{1mm}\\
    \dz \Phi|_{\Gamma_1} = \theta,  \vspace{1mm}\\
    \fint_{\O}\Phi=0,
    \end{array}
    \right.
    \end{equation}
    on the domain $\O:=\T^2\times(0,1)$ with upper boundary $\G_1:=\T^2\times\{1\}$
    and lower boundary $\G_0:=\T^2\times\{0\}$, for any $k\in \N$.
\end{thm}
In the natural way, we say that $\Phi$ is a \emph{$\cal{k+2}$-classical solution} of \eqref{eq:BVP_zeromean}
if $\Phi\in\cal{k+2}(\bar{\O})$ solves the PDE pointwise everywhere on $\O$, and its restriction to the boundary $\dd\O$ satisfies the given boundary conditions. The zero-mean requirement on $\Phi$ is prescribed to guarantee uniqueness of solution to the boundary value problem. Associated to this solution map is the operator defined below, which determines the activity in the surface semi-geostrophic equation \eqref{eq:SSG}.
\begin{de}\label{def:berp}
The {\bf Neumann-to-Dirichlet operator} $\mathcal T$ associated to system \eqref{eq:SSG} 
is defined to be
\begin{equation}\label{eq:def_R_S}
\mathcal{T}:=\mathcal R\circ\mathcal S:\cal{k+1}(\T^2)\to \cal{k+2}(\T^2),
\end{equation}
where $\mathcal R:\cal{k}(\bar{\O})\to \cal k (\T^2)$
is the classical restriction operator to the set $\mathbb{T}^{2}\times \{1\}$.
\end{de}
The main idea of the proof of theorem \ref{thm:nonlin_BVP}, and thereby the demonstration that $\mathcal{T}$ is well defined, is to employ the Banach Fixed Point theorem to guarantee both existence and uniqueness of solutions
to \eqref{eq:BVP_zeromean}, inspired by the numerical work \cite{BR16:turbulence} by Badin and Ragone on SSG. Indeed, for a given $\ta \in\cal{k+1}(\T^2)$ we define the operator
$T^{(\ta)}:\cal {k+2}(\bar{\O})\to\cal {k+2}(\bar{\O})$ as
\begin{equation}\label{eq:def_Tta}
	T^{(\ta)}:= \mathcal A_{\ta}\circ \mathcal M,
\end{equation}
where the Monge-Amp\`ere-type operator $\mathcal M:\cal {k+2}(\bar{\O})\to\cal k(\bar{\O})$ is defined as
\begin{equation*}
	\mathcal M[\phi] := \dxx \phi \dyy\phi - (\dxy\phi)^2,
\end{equation*}
and the operator $\mathcal A_{g}:\cal {k}(\bar{\O})\to\cal {k+2}(\bar{\O})$ is simply the solution operator $\mathcal A_g:f\mapsto u$ associated to the linear Neumann boundary value problem on $\O$ given by
\begin{equation}\label{eq:BVP_Poisson}
    \left\{
    \begin{array}{l}
        \Delta u = f,  \vspace{1mm}\\
    \dz u|_{\G_0}=0,  \vspace{1mm}\\
    \dz u|_{\G_1}=g,  \vspace{1mm}\\
    {\fint_{\O}u = 0,}
    \end{array}
    \right.
  \end{equation}
for any given $g\in\cal{k+1}(\T^2)$. It follows that a fixed point of $T^{(\ta)}$ is a classical solution 
of \eqref{eq:BVP_zeromean} on $\O$ of class $\cal{k+2}$.

The fact that all boundary value problems under study in the sequel are, roughly speaking, periodic in two coordinate directions and non-periodic but of bounded extent in the other makes their analysis slightly awkward. For instance, in the case of the pure Poisson Neumann boundary value problem, one cannot simply apply routine techniques from \cite{GT:ellipticPDEs} to understand their well-posedness. One could, for instance, employ techniques from the monograph \cite{Grubb:book} in order to employ a Method of Reflections-type argument for elliptic equations on the unbounded strip $\Omega\subset\mathbb{R}^{3}$. This is not the approach we adopt in this work, however.

We begin our approach to the proof of Theorem \ref{thm:nonlin_BVP} by showing well-posedness of the linear BVP
\eqref{eq:BVP_Poisson} in Section \ref{sec:linBVP}. Following this, we look to verify the hypotheses of the Banach Fixed Point Theorem
to the operator $T^{(\ta)}$ on a suitable ball in $\cal{k+2}(\bar{\O})$ in Section \ref{thm:nonlin_BVP} thereafter.

\subsection{Analysis of the Poisson Problem \texorpdfstring{\eqref{eq:BVP_Poisson}}{(3.5)}}\label{sec:linBVP}
We begin by recalling the necessary compatibility condition that needs
to be satisfied by the inhomogeneity $f$ and the boundary datum $g$
for the Poisson problem \eqref{eq:BVP_Poisson} be well posed.
  
  \begin{lem}[Compatibility condition] \label{pr:compat_cond}
    A necessary condition for the problem \eqref{eq:BVP_Poisson} to have
    a classical $C^2$-solution on $\Omega$ is that $f$ and $g$ are
    compatible in the sense that
    \begin{equation*}
    \int_{\Omega} f = \int_{\T^2} g .
    \end{equation*}
  \end{lem}

  The well-posedness of the system \eqref{eq:BVP_Poisson}, and hence the 
  well-posedness of the operator $\mathcal A_{g}$, is proved by means of 
  Schauder theory. Let us first state the various notions of solution to
  the Poisson problem with which we work in the sequel.
  
  \begin{de}
  Given $f\in L^2_{\loc}(\Omega)$ and $g\in H^1_{\loc}(\T^2)$, 
  a function $u$ is said to be:
  \begin{itemize}
  \item a \emph{weak solution} of \eqref{eq:BVP_Poisson} if $u\in H^1_{\loc}(\Omega)$ and 
      \begin{equation*}
	-\int_{\Omega} \nabla u\cdot \nabla \phi + \int_{\T^2} g \phi|_{\G_1} 
	  = \int_{\Omega} f\phi \qquad \forall \phi\in \ccinf_c(\bar{\O});
      \end{equation*}
  \item a \emph{strong solution} of \eqref{eq:BVP_Poisson} if $u\in H_{\loc}^2(\Omega)$ and
      \begin{equation*}
      \Delta u = f \quad \text{in }L_{\loc}^2(\Omega) \qquad \text{and}\qquad 
	\text{Tr}_{1} [\partial_{x_3} u] = g,
	  \ \ \text{Tr}_{0} [\partial_{x_3} u] = 0 \quad \text{in }L_{\loc}^2(\T^2),
      \end{equation*}
      where $\Tr_i$ is the trace operator 
      $\Tr_i:H^1(\Omega)\to L^2(\G_i)$, for $i\in\{0,1\}$. 
  \end{itemize}
  \end{de}
We now state the expected result on existence and uniqueness of classical solutions of the Poisson problem \eqref{eq:BVP_Poisson}.
  \begin{thm}[Existence and uniqueness of classical solutions to the Poisson problem]
  \label{thm:WP_BVP_Poisson}
   Given $f\in\cal0(\bar{\O})$ and compatible $g\in\cal1 (\T^2)$,
   there exists a unique classical solution $u\in\cal2(\bar{\O})$
   of \eqref{eq:BVP_Poisson}. More generally, for $k\in\N$, if $f\in\cal k (\bar {\O})$ and $g\in\cal{k+1}(\T^2)$, 
   then $u\in\cal{k+2}(\bar{\O})$.
   Furthermore, there exists a constant $\widetilde C_{k,\al}>0$ such that 
   \begin{equation}\label{eq:estimate_BVP_Poisson}
    \|u\|_{\cal{k+2}(\bar{\O})} \leq \widetilde C_{k,\al}\left(\|f\|_{\cal k(\bar{\O})}+ \|g\|_{\cal{k+1}(\T^2)}\right).
   \end{equation}
  \end{thm}

  \begin{proof} 
  We proceed by first demonstrating the  existence of weak solutions of the Poisson problem, and then in turn proving higher regularity thereof. Following this, we demonstrate the validity of the Schauder estimate \eqref{eq:estimate_BVP_Poisson}.
  
{\em I. Existence of a Weak Solution.} For the existence of weak solutions of the Poisson boundary value problem, we consider the following two auxiliary problems:
 \begin{equation}
   \label{eq:auxil_BVP_vw}
   \left\{
   \begin{array}{l}
    \Delta v = f,  \vspace{1mm} \\ 
    \dz v|_{\G_0}=0,  \vspace{1mm}\\
    \dz v|_{\G_1}=0,  \vspace{1mm}\\
    { \fint_{\Omega} v =0},
   \end{array}\right.
   \hspace{2cm}
    \left\{
    \begin{array}{l}
    \Delta w = 0,  \vspace{1mm}\\
    \dz w|_{\G_0}=0,  \vspace{1mm}\\
    \dz w|_{\G_1}= g,  \vspace{1mm}\\
    { \fint_{\Omega}w =0}.
    \end{array}\right.
    \end{equation}
    The existence of a weak solution $v\in H^1(\O)$ to
    the first problem above is guaranteed by a standard application of the Lax-Milgram Theorem 
    (see for instance \cite[\S~5]{GT:ellipticPDEs}).
    For the solution $w$ of the second problem, the aim is to show that the 
    function formally defined on $\O$ by
    \begin{equation*}
	    w(x',x_3)=\sum_{\substack{k\in\Z^2\\k\neq(0,0)}} 
		\frac{\hat g_k \cosh(2\pi |k|x_3)}{2\pi|k|\sinh(2\pi|k|)}\, 
		e^{2\pi i k\cdot x'} \qquad \forall (x',x_3)\in\Omega,
    \end{equation*}
    is a weak solution of the Laplace problem, where $x':=(x_1,x_2)$ and 
    $\hat g_k$ is the $k$\textsuperscript{th} Fourier coefficient of $g$ for $k\in\mathbb{Z}^{2}$.
    By the Cauchy-Schwarz inequality, Plancherel lemma and Maclaurin-Cauchy test for series, it follows that $w$ admits the bound
    \begin{align}\label{eq:C0_bound_w}
        \|w\|_{\linf(\O)}&\leq\frac{\coth(2\pi)}{2\pi} 
            {\sqrt{ \sum_{k\in\Z^2} |k|^2 |\hat g_k|^2}}
            {\sqrt{\sum_{\substack{k\in\Z^2 \\ k\neq (0,0)}}\frac{1}{|k|^4} }} \notag\\
            &\leq \frac{\coth (2\pi)}{2\sqrt{\pi}}\|g\|_{H^1(\T^2)}.
    \end{align}
    In particular, $w$ is obtained as 
    the weak limit in $H^2(\O)$ of the sequence 
    $\{w_M\}_{M\in\N}$ of truncated sums, namely
    \begin{equation*}
    w_M(x',x_3):= \sum_{\substack{k\in\Z^2\\0<|k|\leq M }} 
        \frac{\hat g_k \cosh(2\pi |k|x_3)}{2\pi|k|\sinh(2\pi|k|)}\, e^{2\pi i k\cdot x'}
        \qquad \forall (x',x_3)\in\T^2\times(0,1), \ \forall M\in\N.
    \end{equation*}
Indeed, it is readily verified that $w_M$ is the classical solution
    of the Laplace problem with Neumann boundary conditions
    $\dz w_M|_{\G_0} = 0$ and 
    $\dz w_M|_{\G_1} = g_M$,
    where $g_M$ is the projection of $g$ onto its set of $M$\textsuperscript{th}-order Fourier modes, namely 
    \begin{equation*}
        g_M(x'):=\sum_{\substack{k\in\Z^2\\0\leq|k|\leq M }}
            \hat g_k e^{2\pi i k\cdot x'}\qquad \forall x'\in\T^2,\ \forall M\in\N.
    \end{equation*}
The sequence $\{w_M\}_{M\in\N}$ is 
    uniformly bounded in $H^2(\O)$, owing to the fact that one can show that 
    \begin{equation*}
    	\|D^2 w_M\|_{L^2(\O)}\leq C\|g\|_{H^1(\T^2)},
    \end{equation*}
    where the constant $C>0$ does not depend on the truncation parameter $M$.
    By a suitable version of Poincar\'e's inequality for functions periodic in only two co-ordinate directions, we conclude that $\{w_{M}\}_{N\in\mathbb{N}}$ is weakly relatively-compact in $H^2(\O)$, and thereby infer the existence of a (relabelled) convergent subsequence to an element in $H^2$. By uniqueness of the strong $L^2$-limit, the $H^2$-weak limit is $w$.
    It is therefore only left to show that $w$ is a weak solution of the Neumann boundary value problem for Laplace's equation in 
    \eqref{eq:auxil_BVP_vw}. Indeed, as we have that
    $w_M\wto w$ in $H^2$, for any $\phi\in\ccinf(\Omega)$ it follows that
    \begin{equation*}
        \int_{\Omega}\Delta w \phi = \lim_{M\to\infty} \int_{\Omega}\Delta w_M \phi =0.
    \end{equation*}
As far as the boundary conditions are concerned, we check that for any $\phi\in\ccinf(\Omega)$, the following holds true:
    \begin{align*}
        0&=\lim_{M\to\infty} \int_{\Omega}\Delta w_M\phi 
        = \lim_{M\to\infty}\left(-\int_{\Omega} \nabla w_M\cdot \nabla \phi 
            + \int_{\Gamma_1} g_M \phi|_{\Gamma_1}\right)\\
        &= -\int_{\Omega}\nabla w\cdot\nabla \phi 
            + \int_{\Gamma_1}g \phi|_{\Gamma_1}
        = \int_{\Gamma_1} \left(g - \Tr_{1}[\dz w]\right)\phi|_{\Gamma_1}.
    \end{align*}
    In particular, for any $\psi\in\ccinf(\T^2)$ we have that
    \begin{equation*}
        \int_{\T^2} \left(g - \Tr_{1}[\dz w]\right)\psi= 0,
    \end{equation*}
    therefore $g = \dz w|_{\Gamma_1}$ in $L^2(\T^2)$,
    and we deduce that $w$ is a weak solution of the given boundary value problem. By the reasoning above, we conclude that the function $u:=v+w\in H^1(\O)$ is a weak solution of the Neumann boundary value problem for the Poisson equation . 
    
{\em II. Higher Regularity of Weak Solution.}  We now show that $u$ is a classical solution of \eqref{eq:BVP_Poisson}, not only a weak solution thereof.
    By elliptic interior regularity 
    (Theorem 9.19 in \cite[\S~9.6]{GT:ellipticPDEs}), 
    we know that $u\in\cal{2}(\O)$.
    By Theorem 5.54 in \cite[\S~5]{Lieberman13:elliptic} 
    applied in the setting of the smooth unbounded domain $\R^2\times(0,1)$,
    the weak solution $u$ belongs also to $\cal 1(\bar{\O})$.
    Hence, $u\in \cal2(\O)\cap\cal1(\bar{\O})$.
    We now consider $u$ as a function defined on $\R^2\times (0,1)$ that is $1$-periodic in the $x$ and $y$ directions, 
    and apply Theorem 6.26 from \cite[\S~6.7]{GT:ellipticPDEs} to show that $u$ is twice differentiable up to the
    upper and lower boundaries of $\Omega$, i.e. $u\in\cal2(\bar{\O})$.
    In the more general case that $f\in\cal k(\bar{\O})$ and $g\in \cal{k+1}(\T^2)$, one can prove by induction
    the Neumann-data equivalent of Theorem 6.19 in \cite[\S~6.4]{GT:ellipticPDEs},
    using Theorem 6.26 in \cite[\S~6.7]{GT:ellipticPDEs} as the first step of the induction argument,
    to show that $u\in\cal{k+2}(\bar{\O})$. 
    Uniqueness of smooth solutions of this BVP follows by standard energy methods applied to differences of candidate solutions.

{\em III. The Schauder Estimate.} We conclude the proof of the theorem by demonstrating the claimed Schauder estimate. It is a well known fact (\cite[Thm. 7.3]{ADN:boundary_est_elliptic}) that 
    \begin{equation*}
    \|u\|_{\cal{k+2}(\bar{\O})} \lesssim \|f\|_{\cal{k}(\bar{\O})}
        + \| g \|_{\cal{k+1}(\T^2)} + \|u\|_{\mathcal C^0(\bar{\O})},
    \end{equation*}
 All that remains is the construction of a bound on the $L^{\infty}$-norm of $u$. 
 We have proved in \eqref{eq:C0_bound_w} that $\|w\|_{\linf(\O)}\lesssim \|g\|_{\cal{k+1}(\T^2)}$. We now derive an analogous bound for the other additive contribution $v$ to $u$, i.e. we show that 
  \begin{equation*}
    \|v\|_{L^{\infty}(\O)}\lesssim \|f\|_{L^2(\O)}.
   \end{equation*}
By the Sobolev Embedding Theorem and the appropriate variant of Poincar\'e's inequality, one has that 
    \begin{equation*}
      \|v\|_{L^{\infty}(\O)}\lesssim \|v\|_{H^2(\O)} \lesssim \|D^2 v\|_{L^2(\O)}.
    \end{equation*}
 Now let $(\phi_n)_{n\in\N}\subset \ccinf(\O)$ be such that 
    $\phi_n\to v$ in $H^2(\O)$ and $\dz \phi_n=0$ on $\Gamma_0\cup \Gamma_1$. 
    Then, for any $n\in\N$, an application of integration by parts yields that
    \begin{align*}
      \|D^2\phi_n\|_{L^2}^2 =\|\Dl \phi_n\|^2_{L^2(\O)}.
    \end{align*}
    Therefore, strong convergence in $H^2(\Omega)$ of the sequence $\{\phi_n\}_{n\in\N}$ allows us to deduce that   
    \begin{align*}
      \|v\|_{L^{\infty}}\lesssim \|D^2 v \|_{L^2} = \lim_{n\to\infty}\|D^2\phi_n\|_{L^2} 
	= \lim_{n\to\infty}\|\Dl\phi_n\|_{L^2} = \|\Dl v\|_{L^2}=\|f\|_{L^2}.
    \end{align*}
This concludes the proof of the theorem.	
  \end{proof}


  \subsection{Proof of Theorem \ref{thm:nonlin_BVP}}

We now aim to prove Theorem \ref{thm:nonlin_BVP}
    using a fixed point argument. We first note that if $\phi\in C^2(\T^2)$, then it holds that
  \begin{equation*}
    \int_{\T^2}\mathcal M[\phi]=0.
  \end{equation*}
By definition of the operator $T^{(\ta)}$, the operator
  $\mathcal A_{\ta}$ is evaluated at $\mathcal M[\Phi]$, which 
  has zero mean and hence the compatibility
  condition for solvability of the Neumann problem \eqref{eq:BVP_zeromean} requires us to define $\mathcal A_{\ta}$ for 
  $\ta \in\cal{k+1}(\T^2)$ such that
  \begin{equation*}
    \int_{\T^2} \ta  = 0.
  \end{equation*}
  Secondly, we note that the Schauder estimate \eqref{eq:estimate_BVP_Poisson} implies the following estimate on the operator $\mathcal A_{g}$:
    \begin{equation}\label{eq:Poisson_estimate}
        \|\mathcal A_g[f]\|_{\cal{k+2}(\bar{\O})}\leq \widetilde C_{k,\al}
            (\|f\|_{\cal k(\bar{\O})}+\|g\|_{\cal{k+1}(\T^2)}),
    \end{equation}
for any $f\in C^{k, \alpha}(\overline{\Omega})$ and $g\in C^{k+1, \alpha}(\mathbb{T}^{2})$. In turn, the following proposition follows from an easy calculation.
\begin{pr}\label{pr:M_est}
    Let $k\in\N$. For any $f_{1}, f_{2}\in C^{k+2,\alpha}(\bar{\O})$, one has that
    \begin{equation*}
    \|\mathcal M[f_{1}]\|_{C^{k,\alpha}(\bar{\O})} \leq 2\|f_{1}\|^2_{C^{k+2,\alpha}(\bar{\O})},
    \end{equation*}
    and
    \begin{equation*}
    \|\mathcal M[f_{1}]-\mathcal M[f_{2}]\|_{C^{k,\alpha}(\bar{\O})} \leq 2 \left(\|f_{1}\|_{C^{k+2,\alpha}(\bar{\O})}
    + \|f_{2}\|_{C^{k+2,\alpha}(\bar{\O})}\right)\|f_{1}-f_{2}\|_{C^{k+2,\alpha}(\bar{\O})}.
    \end{equation*}
\end{pr}
With these observations in hand, we now present the proof of the main theorem of this section. 

\begin{proof}[Proof of Theorem \ref{thm:nonlin_BVP}]
    For any $k\in\N$, consider the ball in $\cal{k+2}(\bar{\O})$ defined by
    \begin{equation*}
        \mathrm X^{(k,\al)}:= \left\{ \phi \in \cal{k+2}(\bar{\O}) :
        \|\phi\|_{\cal{k+2}} \leq R_1^{(k,\al)}, \int_{\O}\phi = 0 \right\},
    \end{equation*}
    and the ball in $\cal{k+1}(\T^2)$ given by
    \begin{equation*}
        \mathrm Y^{(k,\al)}:= \left\{ \ta\in\cal {k+1}(\T^2) :
            \|\ta\|_{\cal{k+1}} \leq R_2^{(k,\al)}, \int_{\T^2}\ta =0 \right\},
    \end{equation*}
    with the radii $R_1^{(k,\al)}$ and $R_2^{(k,\al)}$ satisfying
    \begin{equation}\label{eq:def_R1R2}
    R_1^{(k,\al)}\leq\frac 1{8\widetilde C_{k,\al}}, \qquad
    R_2^{(k,\al)}\leq \min \left\{\frac1{8 \widetilde C_{k,\al}^2}, \frac{R_1^{(k,\al)}}{\widetilde C_{k,\al}} - 2 (R_1^{(k,\al)})^2 \right\}.
    \end{equation}
We proceed by steps, showing (i) existence and uniqueness of solutions
    to \eqref{eq:BVP_zeromean}, (ii) continuity of the solution operator $\mathcal S$, and then (iii) its Fr\'echet 
    differentiability.

    \noi 
{\em I. Existence and Uniqueness.} In order to prove existence and uniqueness of solutions in $\cal{k+2}(\bar{\O})$
    we use a Banach fixed point argument by showing that the 
    operator $T^{(\ta)}:\mathrm X^{(k,\al)}\to\cal{k+2}(\bar{\O})$ defined in \eqref{eq:def_Tta}
    is a contraction on $\mathrm X^{(k,\al)}$. The estimate \eqref{eq:Poisson_estimate}
    on the operator $\mathcal A_g$ and 
    Proposition \ref{pr:M_est} imply that for any $\phi\in \mathrm X^{(k,\al)}$, one has that
    \begin{equation*}
        \|T^{(\ta)}[\phi]\|_{\cal{k+2}} \leq R_1^{(k,\al)},
    \end{equation*}
while for any $\phi,\psi\in \mathrm X^{(k,\al)}$, it holds that
    \begin{equation*}
        \|T^{(\ta)}[\phi]-T^{(\ta)}[\psi]\|_{\cal{k+2}}\leq \frac12\|\phi-\psi\|_{\cal{k+2}}.
    \end{equation*}
In addition, $T^{(\ta)}[\phi]$ has zero mean over $\mathbb{T}^{2}\times (0, 1)$. Therefore, it follows that $T^{(\ta)}$ is a contraction on the complete metric space $\mathrm X^{(k,\al)}$,
    and by the Banach Fixed Point Theorem we deduce the existence of a unique fixed point $\Phi$ 
    for $T^{(\ta)}$ in $\mathrm X^{(k,\al)}$. Owing to this observation, we define the operator $\mathcal{S}: \mathrm Y^{(k,\al)}\to \mathrm X^{(k,\al)}$ pointwise by
    $\mathcal S[\ta]:=\Phi$, where $\Phi$ is the unique fixed point of $T^{(\ta)}$ in $\mathrm X^{(k,\al)}$.
   
   \noi 
{\em II. Continuity of the Operator $\mathcal S$.}
    Using Schauder estimates the reader can  verify that $\mathcal S$ is continuous in the natural H\"older topologies. 

    \noi
{\em III. Fr\'{e}chet Differentiability of the Operator $\mathcal S$.} Let us now demonstrate the Fr\'echet differentiability of 
    $\mathcal S:\mathrm Y^{(k,\al)}\to \mathrm X^{(k,\al)}$.
    For any $\ta\in \mathrm Y^{(k,\al)}$, define the linear  operator
    $\mathcal L[\ta]:\mathrm Y^{(k,\al)} \to \cal{k+2}(\bar{\O})$ 
    as the solution operator on $\mathrm X^{(k,\al)}$
    associated to the following linear boundary value problem:
    \begin{equation*}
        \left\{
        \begin{array}{l}
        \Dl \mathcal L[\ta](h) = \g(\mathcal L[\ta](h),\mathcal S[\ta]), \vspace{1mm}\\
        \dz \mathcal L[\ta](h) |_{\G_0}=0, \vspace{1mm}\\
        \dz \mathcal L[\ta](h)|_{\G_1}= h, \vspace{1mm}\\
        \int_{\O}\mathcal L[\ta](h)=0,
        \end{array}\right.
    \end{equation*}
    where the bilinear symmetric map 
    $\g:\cal{k+2}(\bar{\O})\times\cal{k+2}(\bar{\O})\to\cal k(\bar{\O})$ 
    is defined as
    \begin{equation*}
        \g(f,g)=\dxx f\dyy g + \dyy f \dxx g - 2 \dxy f\dxy g.
    \end{equation*}
One notes that for $f_{1}, f_{2}\in\cal{k+2}(\bar{\O})$, it holds that
    \begin{equation}\label{eq:com_g_norm}
        \|\g(f_{1}, f_{2})\|_{\cal k}\leq 4\|f_{1}\|_{\cal{k+2}}\|f_{2}\|_{\cal{k+2}}
    \end{equation}
    and in turn that
    \begin{align}\label{eq:bound_Frech_der}
        \|\mathcal L[\ta](h)\|_{\cal{k+2}}\leq 2\widetilde C_{k,\al}\|h\|_{\cal{k+1}},
    \end{align}
 which is a consequence of the Schauder estimate \eqref{eq:estimate_BVP_Poisson}. Let us now show that $\mathcal L[\ta]$ is the Fr\'echet derivative of $\mathcal S$.
    For any $\ta, h\in \mathrm Y^{(k,\al)}$ such that
    $\ta+h\in \mathrm Y^{(k,\al)}$, we consider the nonlinear 
    operator $\mathcal N[\ta]:h\mapsto \mathcal S[\ta + h] 
    - \mathcal S[\ta] - \mathcal L[\ta](h)$,
    which is the solution operator associated to the following BVP:
    \begin{equation*}
        \left\{
        \begin{array}{l}
        \Dl \mathcal N[\ta](h)  = \mathcal M[\mathcal L[\ta](h)] 
            + \mathcal M[\mathcal N[\ta](h)] 
            + \g\left(\mathcal N[\ta](h) ,\mathcal L[\ta](h) +\mathcal S[\ta]\right), \vspace{1mm}\\
        \dz \mathcal N[\ta](h)|_{\G_0\cup \G_1}=0, \vspace{1mm}\\
        \int_{\O}\mathcal N[\ta](h)  =0.
        \end{array}\right.
    \end{equation*}
    The above is proved by considering $\Delta \mathcal S[\theta+h]-\Delta \mathcal S[\theta]$, rewriting $     \mathcal S[\theta+h]=\mathcal S[\theta] + \mathcal L[\theta](h)+\mathcal N[\theta](h)$ and considering the equation for $ c\mathcal L[\theta](h)$.
In order to claim that $\mathcal L[\ta]$ is the Fr\'echet derivative of $\mathcal S$ at $\theta$, it remains to show that 
    \begin{equation*}
    	\lim_{h\to 0} \frac{\|\mathcal N[\ta](h)\|_{\cal{k+2}}}{\|h\|_{\cal{k+1}}}=0.
    \end{equation*}
    By the estimates \eqref{eq:estimate_BVP_Poisson}, \eqref{eq:com_g_norm} 
    and Proposition \ref{pr:M_est}, the norm of the operator $\mathcal N[\ta]$ admits the bound given by
    \begin{align*}
        \|\mathcal N[\ta](h)\|_{\cal{k+2}} 
        &\leq 2\widetilde C_{k,\al} \|\mathcal L[\ta](h)\|_{\cal{k+2}}^2
            + 2\widetilde C_{k,\al} 
             \|\mathcal N[\ta](h)\|_{\cal{k+2}}^2 \\
        &\qquad + 4\widetilde C_{k,\al}\|\mathcal N[\ta](h)\|_{\cal{k+2}} \|\mathcal S[\ta]+\mathcal L[\ta](h)\|_{\cal{k+2}}.
    \end{align*}
 This estimate can be rewritten simply as 
    \begin{align*}
        \beta\|\mathcal N[\ta](h)\|_{\cal{k+2}}
        \leq 2\widetilde C_{k,\al}  \|\mathcal L[\ta](h)\|^2_{\cal{k+2}},
    \end{align*}
    where $\beta:=1 - 2\widetilde C_{k,\al}\|\mathcal N[\ta](h)\|_{\cal{k+2}} 
            -4\widetilde C_{k,\al} \|\mathcal S[\ta]+\mathcal L[\ta](h)\|_{\cal{k+2}}$.
    As $h\to 0$ in $\mathcal{C}^{k+1, \alpha}(\overline{\Omega})$, the quantity $\beta$ tends 
    to $1-4\widetilde C_{k,\al} \|\mathcal S[\ta]\|_{k+2,\al}{\geq} \frac12$,
    by continuity of the operator $\mathcal S:\mathrm Y^{(k,\al)}\to\mathrm X^{(k,\al)}$.
    Hence, if $h$ is taken to be in a sufficiently-small ball around the zero map, it follows that $\beta\geq \frac14$ and
    using the bound \eqref{eq:bound_Frech_der} on the norm of $\mathcal L[\ta]$,
    we infer that
    \begin{align*}
        \frac{\|\mathcal N[\ta](h)\|_{\cal{k+2}}}{\|h\|_{\cal{k+1}}}
        \leq 8\widetilde C_{k,\al} \frac{\|\mathcal L[\ta](h)\|^2_{\cal{k+2}}}{\|h\|_{\cal{k+1}}}
        \leq32 \widetilde C_{k,\al}^3\|h\|_{\cal{k+1}}.
    \end{align*} 
    This concludes the proof of Theorem \ref{thm:nonlin_BVP}.
  \end{proof}
  
  \begin{rmk}
  	By an application of the Mean Value Theorem 
  	for maps with range in a Banach space (see Theorem 4 in \cite[\S~3.2]{Cheney:Banach_ana}),
  	one may conclude that the operator $\mathcal S$ is globally Lipschitz on 
  	$\mathrm Y^{(k,\al)}$ with Lipschitz constant $2\widetilde C_{k,\al}$.
  \end{rmk}

  \begin{rmk}
  	By utilising a similar argument, Theorem \ref{thm:nonlin_BVP} can be established
  	in the context of Sobolev spaces as opposed to H\"{o}lder spaces, i.e. there exists a Fr\'echet-differentiable map $\mathcal S$ defined on a closed ball in $H^{k+1}(\T^2)$
  	and valued in $H^{k+2}(\O)$ such that
  	$\mathcal S[\ta]$ is a $H^{k+2}$-strong solution of \eqref{eq:BVP_zeromean},
  for $2\leq k\in\N$. Observe that we require $k\geq 2$ in order to guarantee that $H^{k+2}$ be a Banach algebra, which is required for the relevant estimates on the Monge-Amp\`ere 
  operator $\mathcal M$.
  \end{rmk}


\section{Existence and Uniqueness of Classical Solutions of SSG}\label{sec:SSG}

By the work of the previous sections, we have constructed the Neumann-to-Dirichlet operator that determines the activity in the active scalar equation that defines SSG, namely equation \eqref{eq:SSG}. However, in order to demonstrate the local-in-time existence of classical solutions of the 
active scalar equation itself, we perform a Schauder fixed point argument. In broad strokes, given a suitable time-dependent and spatially-incompressible vector field $w$, we first construct a solution $\alpha_{w}$ to the IVP associated to the {\em passive} scalar equation given by
\begin{equation*}
\partial_{t}\alpha_{w}+(w\cdot\nabla)\alpha_{w}=0
\end{equation*}
with given fixed initial datum $\alpha_{w}(\cdot, 0):=\theta_{0}$ that is independent of $w$, and in turn build a new time-dependent
vector field $\widetilde{w}(\cdot, t):= \nabla^{\perp}\mathcal{T}[\alpha_{w}(\cdot, t)]$ pointwise in time by way of the Neumann-to-Dirichlet map $\mathcal{T}$. Any fixed point of the
operator $\mathcal{W}:w\mapsto\widetilde{w}$ (in an appropriate space) furnishes a classical solution to 
the active scalar equation \eqref{eq:SSG}. 
 
While the claimed result, Theorem \ref{thm:SSG}, holds true for any integer $k\geq 0$, for the sake of clarity of presentation we consider only the argument pertaining to the local-in-time existence of classical solutions 
     $\ta\in\cbe0([0,\tau_0];\cal 1(\T^2))\cap \mathcal C^1((0,\tau_0)\times\T^2)$ of the active scalar equation with $\ta_0\in\cal2(\T^2)$, for some $\tau_{0}>0$ depending on the initial datum, $\alpha$ and $\beta$. A similar argument can be employed to prove the existence of 
    solutions $\ta\in\cbe0([0,\tau_k];\cal {k+1}(\T^2))\cap \mathcal C^1((0,\tau_k);\mathcal C^{1}(\T^2))$ with initial data  
    $\ta_0\in\cal {k+2}(\T^2)$, for $k\geq 1$ and time of existence $\tau_{k}>0$, the details of which we omit here. We shall, however, include the proof of uniqueness of classical solutions in the case that the initial datum $\theta_{0}$ lies in a suitable small ball in $\mathcal{C}^{k, \alpha}(\mathbb{T}^{2})$ for $k\geq 4$.
    
\begin{proof}[Proof of Theorem \ref{thm:SSG}]
Firstly we show existence of smooth solutions, and later we prove uniqueness. 
\noi
{\em I. Existence of the dynamics.}
Suppose that $\ta_0\in \cal2(\T^2)$ with 
    \begin{equation}\notag
    \int_{\T^2}\ta_0 = 0 , \hspace{1cm} \|\ta_0\|_{\cal2}\leq \rho\leq\min\left\{ \frac{R_2^{(0,\al)}}{8},\frac{R_1^{(0,\al)}}{48\widetilde C_{(0,\al)}}\right\},
    \end{equation}
    where $R_1^{(0,\al)}$ and $R_2^{(0,\al)}$ are defined in \eqref{eq:def_R1R2}, and fix $\tau_{0}>0$ to be any positive number such that
\begin{equation}\notag
0<\tau_{0}<\min\left\{
\frac{\ln 2}{(2+\alpha)R_1^{(0, \alpha)}}, (R^{(0, \alpha)}_{1})^{-\frac{1}{1-\beta}}
\right\}.
\end{equation} 
Let $0<\alpha'<\alpha<1$ and $0<\beta'<\beta<1$, and consider the topological space $(\mathcal{D}, \mathsf{T}_{(\alpha', \beta')})$, where
\begin{equation}\notag
\mathcal{D}:=\left\{
w\in C^{0, \beta}([0, \tau_{0}]; C^{1, \alpha}_{\sigma}(\T^{2}; \R^{2}))\,:\,\|w\|_{C^{0, \beta}_{t}C^{1, \alpha}_{x}}\leq R_{1}^{(0, \alpha)}
\right\}
\end{equation}
and $\mathsf{T}_{(\alpha', \beta')}$ is the topology on $\mathcal{D}$ generated by the $\|\cdot\|_{C^{0, \beta'}_{t}C^{1, \alpha}_{x}}$-norm. The notation $C^{1, \alpha}_{\sigma}(\T^{2}; \R^{2}))$ is here used to denote the space of elements of $C^{1, \alpha}(\T^{2}; \R^{2}))$ that are divergence-free pointwise everywhere on $\T^2$. 
One may check that the topological space $(\mathcal{D}, \mathsf{T}_{(\alpha', \beta')})$ is convex, compact and closed. Our principle operator of interest $\mathcal{W}:\mathcal{D}\rightarrow C^{0, \beta}([0, \tau_{0}]; C^{1, \alpha}_{\sigma}(\mathbb{T}^{2}, \mathbb{R}^{2}))$ is defined pointwise as
\begin{equation*}
(\mathcal{W}[w])(\cdot, t):=\nabla^{\perp}\mathcal{R}\circ \mathcal{S}[\theta_{w}(\cdot, t)]
\end{equation*}
for all $t\in [0, \tau_{0}]$, where $\theta_{w}$ denotes the unique classical solution of the IVP given by
\begin{equation*}
\left\{
\begin{array}{l}
\partial_{t}\theta_{w}+(w\cdot\nabla)\theta_{w}=0 \quad \text{on}\hspace{2mm}\mathbb{T}^{2}\times (0, \tau_{0}), \vspace{2mm}\\
\theta_{w}(\cdot, 0)=\theta_{0}.
\end{array}
\right.
\end{equation*}
In order to demonstrate that $\mathcal{W}$ is well defined, we need to verify that $\theta_{w}(\cdot,t)\in \mathrm Y^{(0,\al)}$ for all $t\in[0,\tau_0]$. By Proposition \ref{app:pr:estimate_ta}, we have that
\begin{equation*}
\|\theta(\cdot, t)\|_{1, \alpha}\leq \|\theta_{0}\|_{2, \alpha}\left(
e^{\|w\|\tau_{0}}+2^{\frac{1-\alpha}{2}}e^{2\|w\|\tau_{0}}+\|w\|\tau_{0}e^{(2+\alpha)\|w\|_{\tau_{0}}}
\right)
\end{equation*}
for all $t\in [0, \tau_{0}]$ and so, by our choice of $\tau_{0}$, it follows that
\begin{equation*}
\|\theta(\cdot, t)\|_{1, \alpha}\leq 8\|\theta_{0}\|_{2, \alpha}\leq 8\rho\leq R_{2}^{(0, \alpha)}.
\end{equation*}
Moreover, by estimate in Proposition \ref{app:pr:estimate_ta}, it follows that $\mathcal{W}[w]\in C^{0, \beta}([0, \tau_{0}]; C^{1, \alpha}_{\sigma}(\mathbb{T}^{2}, \mathbb{R}^{2}))$. Now, in order to apply the Schauder Fixed Point Theorem to the operator $\mathcal{W}$, we must verify that both $\mathcal{W}(\mathcal{D})\subset\subset\mathcal{D}$ and that $\mathcal{W}$ is continuous w.r.t. the topology $\mathsf{T}_{(\alpha', \beta')}$ on $\mathcal{D}$. Note that for the former condition, we need only show that $\mathcal{W}(\mathcal{D})\subseteq\mathcal{D}$ owing to the fact that $(\mathcal{D}, \mathsf{T}_{(\alpha', \beta')})$ is compact.

Let us first demonstrate that $\mathcal{W}(\mathcal{D})\subset\mathcal{D}$. Suppose that $w\in \mathcal{D}$. One can check by using the estimates that
\begin{align*}
    \|\mathcal{W}[w]\|_{\mathcal C^{0, \beta}_{t}\mathcal C^{1, \alpha}_{x}} 
    & \leq 2\widetilde{C}_{0,\al}\|\theta_{w}\|_{C_{t}^{0, \beta}C^{1, \alpha}_{x}} \vspace{2mm} \\ 
    &\leq  2\widetilde{C}_{0,\al}\|\ta_0\|e^{\|w\|\tau}
    \big\{1 + 2^{\frac{1-\al}2}e^{\|w\|\tau} + \|w\|\tau e^{(1+\al)\|w\|\tau}\big\}\\
    &\quad + 4\widetilde{C}_{0,\al}\|\ta_0\| \|w\|\tau^{1-\be} e^{\|w\|\tau} 
    \big\{1+e^{\al\|w\|\tau} + 2^{\frac{1-\al}2}e^{\|w\|\tau} + \|w\|\tau e^{(1+\al)\|w\|\tau}\big\}\vspace{2mm}\\
    &\leq 48\widetilde C_{0,\al} \|\ta_0\|\vspace{2mm}\\
    &\leq R_1^{(0,\al)},
    \end{align*}
whence $\mathcal{W}[w]\in\mathcal{D}$. It remains to verify that the map $\mathcal{W}$ is continuous in the topology specified above. It is enough to show that $\mathcal{W}$ is sequentially continuous. As such, let $\{w_{j}\}_{j=1}^{\infty}\subset\mathcal{D}$ be a sequence that converges in $(\mathcal{D}, \mathsf{T}_{(\alpha', \beta')})$, i.e. $w_{j}\rightarrow w$ as $j\rightarrow \infty$ for some $w\in\mathcal{D}$. Let $\{\theta_{w_{j}}\}_{j=1}^{\infty}$ denote the associated sequence of classical solutions to the family of IVPs given by
\begin{equation*}
\left\{
\begin{array}{l}
\partial_{t}\theta_{w_{j}}+(w_{j}\cdot\nabla)\theta_{w_{j}}=0, \vspace{2mm}\\
\theta_{w_{j}}(\cdot, 0)=\theta_{0}.
\end{array}
\right.
\end{equation*}
As we have noted above, the sequence $\{\theta_{w_{j}}\}_{j\in\N}$ is uniformly bounded in $\mathcal{C}^{0, \beta}([0, \tau_{0}]; \mathcal{C}^{1, \alpha}(\mathbb{T}^{2}))$, and therefore there exists a subsequence $\{\theta_{w_{j(k)}}\}_{k=1}^{\infty}\subseteq\{\theta_{w_{j}}\}_{j=1}^{\infty}$ which one can show converges to $\theta_{w}$ in $\mathcal{C}^{0, \beta'}([0, \tau_{0}]; \mathcal{C}^{1, \alpha'}(\mathbb{T}^{2}))$. Owing to the fact that all subsequences of $\{\theta_{w_{j}}\}_{j=1}^{\infty}$ converge to $\theta_{w}$ by uniqueness of classical solutions of the limiting initial-value problem for the transport equation associated to $w$, it follows that $\mathcal{W}$ is continuous. The existence of a local-in-time classical solution of the IVP for the active scalar equation follows readily from an application of Schauder's Fixed Point Theorem.

\noi
{\em II. Uniqueness of the dynamics.}    
  Let $w,v\in\cbe0([0,\tau];\cal3_{\s}(\T^2;\R^2))$ be two vector fields
  and $\ta, \psi\in\cbe0([0,\tau];\cal3(\T^2))\cap C^1((0,\tau)\times\T^2)$ be
  such that $(\ta,w)$ and $(\psi,v)$ are solutions of the active scalar equation 
  \eqref{eq:SSG} with initial condition 
  $\ta(\cdot,0)=\psi(\cdot,0)=\ta_0\in\cal4(\T^2)$.
  We prove that $\|\ta(\cdot,t)-\psi(\cdot,t)\|_{\cal1}=0$
  by demonstrating that 
    \begin{equation}\label{eq:uniqueness}
        \|\ta(\cdot,t)-\psi(\cdot,t)\|_{\cal1}\lesssim  \int_0^t \|\ta(\cdot,r)-\psi(\cdot,r)\|_{\cal1}\,dr.
    \end{equation}
 The claimed uniqueness result then follows by Gr\"onwall's inequality.

 Let us denote the flow map associated with $v$ by $G$ and that associated to $w$ by $F$. In what follows, for notational brevity we write $\|w\|:=\|w\|_{\cbe0_t\cal3_x}$ and 
  $\|\ta_0\|:=\|\ta_0\|_{\cal4}$. We record the estimates below.  
  \begin{itemize}
    \item [(i)] We start noticing that 
        $|F(t,s,x)-G(t,s,x)|\lesssim \int_s^t \|\ta(\cdot,r)-\psi(\cdot,r)\|_{\cal1}    \,dr$. In fact, an application of the Fundamental Theorem of Calculus provides that
         \begin{align*}
        	|F(t,s,x)-G(t,s,x)| &= \bigg| \int_s^t w(F(r,s,x),r)-v(G(r,s,x),r) \,dr \bigg|\\
                & \leq \|w\| \int_s^t |F(r,s,x)-G(r,s,x)|\,dr + \int_s^t \|w(\cdot,r) - v(\cdot,r)\|_{\cal1}\,dr\\
                & \leq \|w\| \int_s^t |F(r,s,x)-G(r,s,x)|\,dr \\
                &\quad+ 2\widetilde C_{0,\al} \int_s^t \|\ta(\cdot,r)-\psi(\cdot,r)\|_{\cal1}\,dr.
        \end{align*}
        By Gr\"onwall's inequality, we infer that
        \begin{align*}
        	|F(t,s,x)-G(t,s,x)| \leq 2\widetilde C_{0,\al} e^{\|w\| |t-s|} \int_s^t \|\ta(\cdot,r)-\psi(\cdot,r)\|_{\cal1}\,dr.
        \end{align*}
    
\item [(ii)] In a similar manner, we prove that 
        $|DF(t,s,x)-DG(t,s,x)|\lesssim \int_s^t \|\ta(\cdot,r)
        -\psi(\cdot,r)\|_{\cal1}\,dr$:
        \begin{align*}
        	|DF&(t,s,x)-DG(t,s,x)| \\
                & = \bigg| \int_s^t Dw(F(r,s,x),r)DF(r,s,x) - Dv(G(r,s,x),r)DG(r,s,x)\bigg|\,dr\\
                & \leq \|w\| \|DF(t,s,\cdot)\|_{\linf} \int_s^t |F(r,s,x)-G(r,s,x)|\,dr\\
                    &\qquad + \|DF(t,s,\cdot)\|_{\linf} \int_s^t \|w(\cdot,r)-v(\cdot,r)\|_{\cal1}\,dr\\
                    &\qquad + \|v\| \int_s^t |DF(r,s,x)-DG(r,s,x)|\,dr\\
                & \leq 2\widetilde C_{0,\al} \|w\||t-s| e^{2\|w\||t-s|} \int_s^t \|\ta(\cdot,r)- \psi(\cdot,r)\|_{\cal1}\,dr\\
                    &\qquad + 2\widetilde C_{0,\al} e^{\|w\||t-s|} \int_s^t \|\ta(\cdot,r)- \psi(\cdot,r)\|_{\cal1}\,dr\\
                    &\qquad + \|v\|\int_s^t |DF(r,s,x)-DG(r,s,x)|\,dr.
        \end{align*}
        Once again, by Gr\"onwall's inequality, we infer that
        \begin{align*}
        	|DF(t,s,x)-DG(t,s,x)|\leq Q \int_s^t \|\ta(\cdot,r)-\psi(\cdot,r)\|_{\cal1}\,dr,
        \end{align*}
        with $Q:=2\widetilde C_{0,\al} e^{(\|w\|+\|v\|)|t-s|}(1+\|w\||t-s|e^{\|w\||t-s|})$.
        
\item [(iii)] Here, we seek to show that $[F(t,s,\cdot)-G(t,s,\cdot)]_{1,\al}\lesssim 
        \int_s^t \|\ta(\cdot,r)-\psi(\cdot,r)\|_{\cal1}\,dr$,
         where we use the notation from \cite{GT:ellipticPDEs} for 
        denoting the H\"older seminorm $[\cdot]_{k,\al}$.
        For a given $\phi\in\cal2(\T^2)$, we are interested in first proving that 
        \begin{equation*}
        	[\phi(F(t,s,\cdot))-\phi(G(t,s,\cdot))]_{0,\al}\lesssim \int_s^t \|\ta(\cdot,r)-\psi(\cdot,r)\|_{\cal1}\,dr.
        \end{equation*}
        Defining $\eta :[0,1]\to \R$ as 
        \begin{align*}
        	\eta(\ld) := \phi(F(t,s, \ld x + (1-\ld)y),t) - \phi(G(t,s, \ld x + (1-\ld)y),t),
        \end{align*}
        it follows that $\eta$ is differentiable and 
        \begin{align*}
        	\phi(F(t,s, x ),t) - \phi(G(t,s, x),t) - \phi(F(t,s, y),t) + \phi(G(t,s, y),t) = \eta(1)- \eta (0) = \int_0^1 \eta'(\ld)\,d\ld.
        \end{align*}
        We focus on $\eta'$, in particular:
        \begin{align*}
            \eta'(\ld) &=  DF(t,s,\ld x + (1-\ld)y)\nabla\phi(F(t,s,\ld x + (1-\ld)y),t) (x-y)\\
                &\qquad -DG(t,s,\ld x + (1-\ld)y) \nabla\phi(G(t,s,\ld x + (1-\ld)y))(x-y)\\
            &=  \big( DF(t,s,\ld x + (1-\ld)y)- DG(t,s,\ld x + (1-\ld)y)\big)\\
            &\qquad\nabla\phi(F(t,s,\ld x + (1-\ld)y),t)(x-y)\\
            &\quad +DG(t,s,\ld x + (1-\ld)y) \\
            &\qquad\big(\nabla\phi(F(t,s,\ld x + (1-\ld)y),t) - \nabla\phi(G(t,s,\ld x + (1-\ld)y),t) \big)(x-y),
        \end{align*}
        hence
        \begin{align*}
        	|\eta'(\ld)| &\leq \|\phi\|_{\cal2} |x-y| |DF(t,s,\ld x + (1-\ld)y)- DG(t,s,\ld x + (1-\ld)y)|\\
                &\qquad + \|DG(t,s,\cdot)\|_{\linf}\|\phi\|_{\cal2} |F(t,s,\ld x + (1-\ld)y) - G(t,s,\ld x + (1-\ld)y)||x-y|\\
        	&\leq Q\|\phi\|_{\cal2}|x-y| \int_s^t \|\ta(\cdot,r)-\psi(\cdot,r)\|_{\cal1}\,dr\\
                &\qquad + 2\widetilde C_{0,\al}e^{(\|w\|+\|v\|)|t-s|}|x-y|\int_s^t \|\ta(\cdot,r)-\psi(\cdot,r)\|_{\cal1}\,dr\\
            & = \sqrt2^{1-\al}\|\phi\|_{\cal2}\left(Q+2\widetilde C_{0,\al} e^{(\|w\|+\|v\|)|t-s|}\right)\int_s^t \|\ta(\cdot,r)-\psi(\cdot,r)\|_{\cal1}\,dr|x-y|^{\al}.
        \end{align*}
        Thus, it follows that
        \begin{align*}
        	|\phi(F(t,s, x ),t) &- \phi(G(t,s, x),t) - \phi(F(t,s, y),t) + \phi(G(t,s, y),t)| \\
        	&=|\eta(1)-\eta(0)|\leq \int_0^1 |\eta'(\ld)|\,d\ld\\
            &= K|x-y|^{\al}\int_s^t \|\ta(\cdot,r)-\psi(\cdot,r)\|_{\cal1}\,dr,
        \end{align*}
        with $K:=\sqrt2^{1-\al}\|\phi\|_{\cal2}\left(Q+2\widetilde C_{0,\al} e^{(\|w\|+\|v\|)|t-s|}\right)$. Now we look at the following H\"older seminorm:
        \begin{align*}
        	|DF&(t,s,x)-DG(t,s,x)-DF(t,s,y)+DG(t,s,y)| \\
        	&= \bigg| \int_s^t Dw(F(r,s,x),r)DF(r,s,x) - Dv(G(r,s,x),r)DG(r,s,x) \\
                &\qquad- Dw(F(r,s,y),r)DF(r,s,y) + Dv(G(r,s,y),r)DG(r,s,y) \,dr\bigg|\\
        	&\leq \int_s^t |Dw(F(r,s,x),r)-Dw(G(r,s,x),r)||DF(r,s,x)-DF(r,s,y)|\,dr\\
                &\qquad + \int_s^t \underbrace{|Dw(F(r,s,x),r) - Dw(G(r,s,x),r)-Dw(F(r,s,y),r) + Dw(G(r,s,y),r)|}_{=\eta(1)-\eta(0) \text{ with }\phi=\dd_iw} \\
                &\qquad \qquad \cdot|DF(r,s,y)|\,dr\\
                &\qquad + \int_s^t |Dw(G(r,s,x),r) - Dv(G(r,s,x),r)| |DF(r,s,x)-DF(r,s,y)|\,dr\\
                &\qquad + \int_s^t |Dw(G(r,s,x),r) - Dv(G(r,s,x),r) -Dw(G(r,s,y),r) + Dv(G(r,s,y),r)|\\
                &\qquad \qquad \cdot|DF(r,s,y)|\,dr\\
                &\qquad + \int_s^t |Dv(G(r,s,x),r) - Dv(G(r,s,y),r)| |DF(r,s,x)-DG(r,s,x)|\,dr\\
            & \leq \|w\|[F(t,s,\cdot)]_{1,\al}|x-y|^{\al}\int_s^t |F(r,s,x)-G(r,s,x)|\,dr\\
                &\qquad + \|DF(t,s,\cdot)\|_{\linf} |\eta(1)-\eta(0)| |t-s|\\
                &\qquad + [F(t,s,\cdot)]_{1,\al}|x-y|^{\al} \int_s^t\|w(\cdot,r)-v(\cdot,r)\|_{\cal1} \,dr\\
                &\qquad + \Lip G(t,s,\cdot)^{\al} |x-y|^{\al}\|DF(t,s,\cdot)\|_{\linf}\int_s^t \|w(\cdot,r)-v(\cdot,r)\|_{\cal1}\,dr\\
                &\qquad + \Lip G(t,s,\cdot)^{\al} |x-y|^{\al} \|v\| \int_s^t |DF(r,s,x)-DG(r,s,x)|\,dr\\
            & \leq \|w\||t-s| 4\widetilde C_{0,\al} e^{2\|w\||t-s|} \int_s^t \|\ta(\cdot,r)-\psi(\cdot,r)\|_{\cal1}\,dr |x-y|^{\al}\\
                &\qquad +  \sqrt2 ^{1-\al} K |t-s| e^{\|w\||t-s|}\int_s^t \|\ta(\cdot,r)-\psi(\cdot,r)\|_{\cal1}\,dr |x-y|^{\al}\\
                &\qquad + 4\widetilde C_{0,\al} e^{\|w\||t-s|}\int_s^t \|\ta(\cdot,r)-\psi(\cdot,r)\|_{\cal1}\,dr |x-y|^{\al}\\
                &\qquad + 2\widetilde C_{0,\al} e^{(\|w\|+ \al\|v\|)|t-s|} \int_s^t \|\ta(\cdot,r)-\psi(\cdot,r)\|_{\cal1}\,dr |x-y|^{\al}\\
                &\qquad + 2\|v\||t-s| e^{\al\|v\||t-s|} Q \int_s^t \|\ta(\cdot,r)-\psi(\cdot,r)\|_{\cal1}\,dr |x-y|^{\al}.
        \end{align*}
        Therefore, 
        \begin{align*}
        	[F(t,s,\cdot)-G(t,s,\cdot)]_{1,\al}\leq \widetilde K \int_s^t \|\ta(\cdot,r)-\psi(\cdot,r)\|_{\cal1}\,dr,
        \end{align*}
        where we use the notation $\widetilde K = \|w\||t-s| 4\widetilde C_{0,\al} e^{2\|w\||t-s|} + \sqrt2 ^{1-\al} K |t-s| e^{\|w\||t-s|} + 4\widetilde C_{0,\al} e^{\|w\||t-s|} + 2\widetilde C_{0,\al} e^{(\|w\|+ \al\|v\|)|t-s|} +  2\|v\||t-s| e^{\al\|v\||t-s|} Q$.
        
\item [(iv)] Let us now show that $\|\ta(\cdot,t)- \psi (\cdot,t)\|_{\linf}\lesssim 
        \int_0^t \|\ta(\cdot,r)-\psi(\cdot,r)\|_{\cal1}\,dr$. Indeed, one finds that
        \begin{align*}
        	|\ta(x,t)-\psi(x,t)| &\leq \|\ta_0\| |F(0,t,x)-G(0,t,x)|\\
        	&\leq 2\widetilde C_{0,\al} e^{\|w\|t}\|\ta_0\| \int_0^t \|\ta(\cdot,r)-\psi(\cdot,r)\|_{\cal1}\,dr,
        \end{align*}
        and so we are done.   
\item [(v)] Let us now consider an analogous estimate for the difference $\|\nabla 
        \ta(\cdot,t) - \nabla\psi(\cdot,t)\|_{\linf}$. One notes that
        \begin{align*}
        	|\nabla \ta(x,t)- \nabla \psi(x,t)| &\leq |DF(0,t,x)-DG(0,t,x)| |\ta_0(F(0,t,x))| \\
                &\qquad + |DG(0,t,x)||\nabla\ta_0(F(0,t,x))-\nabla\ta_0(G(0,t,x))|\\
            &\leq Q \|\ta_0\|  \int_0^t \|\ta(\cdot,r)-\psi(\cdot,r)\|_{\cal1}\,dr\\
                &\qquad + e^{\|v\|t}\|\ta_0\| |F(0,t,x)-G(0,t,x)|\\
            &= \|\ta_0\| \left(Q + 2\widetilde C_{0,\al} e^{(\|w\|+\|v\|)t}\right)  \int_0^t \|\ta(\cdot,r)-\psi(\cdot,r)\|_{\cal1}\,dr.
        \end{align*}
    
\item [(vi)] Finally, let us consider the seminorm term 
        $[\ta(\cdot,t)-\psi(\cdot,t)]_{1,\al}$. One finds that
        \begin{align*}
        	|\nabla \ta(x,t) -& \nabla \psi(x,t) - \nabla \ta(y,t)\ + \nabla \psi(y,t)|\\
        	& = |DF(0,t,x)\nabla \ta_0(F(0,t,x)) - DG(0,t,x)\nabla\ta_0(G(0,t,x)) \\
                &\qquad -DF(0,t,y)\nabla \ta_0(F(0,t,x)) + DG(0,t,y)\nabla\ta_0(G(0,t,y))|\\
        	&= \big|\left( DF(0,t,x)-DG(0,t,x)\right) \nabla\ta_0(F(0,t,x))\\
                &\qquad + DG(0,t,x) \left( \nabla\ta_0(F(0,t,x)) - \nabla\ta_0(G(0,t,x)) \right)\\
                &\qquad +\left( DF(0,t,y)-DG(0,t,y)\right) \nabla\ta_0(F(0,t,y))\\
                &\qquad + DG(0,t,y) \left( \nabla\ta_0(F(0,t,y)) - \nabla\ta_0(G(0,t,y)) \right)\big|\\
        	&\leq |DF(0,t,x)-DG(0,t,x)| |\nabla\ta_0(F(0,t,x))- \nabla\ta_0(F(0,t,y))|\\
                &\qquad + |DF(0,t,x)-DG(0,t,x)-DF(0,t,y)+DG(0,t,y) |\nabla\ta_0(F(0,t,y))||\\
                &\qquad + |DG(0,t,x)-DG(0,t,y)| |\nabla\ta_0 (F(0,t,x)) - \nabla\ta_0(G(0,t,x))|\\
                &\qquad + |DG(0,t,y)| \\
                &\qquad \quad\underbrace{|\nabla\ta_0 (F(0,t,x))-\nabla\ta_0(G(0,t,x))-\nabla\ta_0(F(0,t,y))+\nabla\ta_0(G(0,t,y))|}
                _{=|\eta(1)-\eta(0)| \text{ with }\phi = \dd_i\ta_0}\\
            &\leq Q\sqrt 2^{1-\al}  \|\ta_0\|e^{\|w\|t}\int_0^t \|\ta(\cdot,r)-\psi(\cdot,r)\|_{\cal1}\,dr|x-y|^{\al}\\
                &\qquad + \|\ta_0\| \widetilde K \int_0^t \|\ta(\cdot,r)-\psi(\cdot,r)\|_{\cal1}\,dr|x-y|^{\al}\\
                &\qquad + \sqrt2^{1-\al} e^{\|v\|t}\|\ta_0\|2\widetilde C_{0,\al} e^{\|w\| |t-s|} \int_s^t \|\ta(\cdot,r)-\psi(\cdot,r)\|_{\cal1}\,dr|x-y|^{\al}\\
                &\qquad + e^{\|v\|t} K \int_s^t \|\ta(\cdot,r)-\psi(\cdot,r)\|_{\cal1}\,dr|x-y|^{\al}.
        \end{align*}
        Therefore, we conclude that
        \begin{equation*}
        	[\ta(\cdot,t)-\psi(\cdot,t)]_{1,\al}\lesssim  \int_0^t \|\ta(\cdot,r)-\psi(\cdot,r)\|_{\cal1}\,dr.
        \end{equation*}
  \end{itemize}
  From the definition of the H\"older norm, (iv), (v) and (vi) imply \eqref{eq:uniqueness}. This ends the proof of existence and uniqueness of classical solution under the stated hypotheses.
 \end{proof}


\section{Closing Remarks}\label{sec:concl}
In this work, we have constructed local-in-time smooth solutions of the surface semi-geostrophic equations. Of course, it is desirable to develop a theory of global-in-time weak solutions of the system, principally due to the expectation that the dynamics generated by SG develops fronts in a finite time. 

One important point upon which we have not touched in this paper is {\em Cullen's Stability Principle} and its applications in the setting of the surface semi-geostrophic equations. Indeed, we have not attempted to show that our local-in-time classical solutions admit the property that they give rise to local-in-time classical solutions of the semi-geostrophic equations in Eulerian co-ordinates. We shall explore this in future work.


\appendix

\section{Auxiliary Results: estimates for the transport equation in H\"older spaces}\label{sec:app}

  In this final section we look at the estimates for the solution to
the passive transport equation 
\begin{equation*}
	\left\{
	\begin{array}{l}
		\dt \ta + w\cdot \nabla \ta =0, \vspace{1mm}\\
		\ta(\cdot,0)= \ta_0,
	\end{array}
	\right.
\end{equation*}
with a given vector field $w\in\cbe0([0,\tau];\cal1(\T^2;\R^2))$
and an initial datum $\ta_0\in\cal2(\T^2)$.

\begin{pr}[Estimates on the solution of the transport equation]
\label{app:pr:estimate_ta}
	Given a bounded $w\in\cbe0([0,\tau];\cal1(\T^2;\R^2))$
    and $\ta_0\in\cal2(\T^2)$, the unique classical solution 
    $\ta\in C^1\left((0,\tau)\times\T^2\right)$
    of the transport equation belongs to $\cbe0\left([0,\tau];\cal1(\T^2)\right)$ and satisfies the following estimates:
    \begin{align}
    \notag
    \|\ta(\cdot,t)\|_{\cal1} &\leq \|\ta_0\|e^{\|w\|t} \left(1 + 2^{\frac{1-\al}2}
        e^{\|w        \|t} + \|w\|te^{(1+\al)\|w\|t}\right)
    	\qquad \forall t\in[0,\tau]\\
    \label{app:eq:estimate_ta}
    \|\ta\|_{\cbe0_t\cal1_x} &\leq  \|\ta_0\|e^{\|w\|\tau}
        \big\{1 + 2^{\frac{1-\al}2}e^{\|w\|\tau} + \|w\|\tau e^{(1+\al)\|w\|\tau}\big\}
        \\ 
    \notag
    &\qquad + 2\|\ta_0\| \|w\|\tau^{1-\be} e^{\|w\|\tau} 
        \big\{1+e^{\al\|w\|\tau} + 2^{\frac{1-\al}2}e^{\|w\|\tau} + \|w\|\tau e^{(1+\al)\|w\|\tau}\big\}.
    \end{align}
    where  $\|w\|:=\|w\|_{\cbe0_t\cal1_x}$ and $\|\ta_0\|:=\|\ta_0\|_{\cal2}$.
\end{pr}

\begin{proof}
We recall that, for a Lipschitz vector field $w$ and a $\mathcal C^1$ initial
datum,
the existence and uniqueness of a classical solution is 
guaranteed by the Cauchy-Lipschitz theory of ODEs, and the solution $\ta$ is given by the formula
\begin{equation*}
	\ta(x,t) := \ta_0(F(0,t,x)),
\end{equation*}
where $F:\R\times\R\times\T^2\to\R^2$ is the
\emph{generalised flow map} associated to $w$ as the solution of the 
initial value problem
\begin{equation*}
\left\{
  \begin{array}{l}
   \dt F(t,s,x) = w(F(t,s,x),t), \vspace{1mm}\\
   F(s,s,x)=x.
  \end{array}\right.
 \end{equation*}
 In order to look at estimates on the solution $\ta$, 
 we need some bound on the generalised flow map $F$. 
 
\begin{itemize}
 \item [(a)] $\|F(t,s,\cdot)\|_{\linf}\leq \diam \T^2 + \|w\||t-s|$.
        By the Fundamental Theorem of Calculus, one has that
        \begin{align*}
            |F(t,s,x)|&=\bigg| x+ \int_s^t w(F(r,s,x),r)\,dr \bigg|\\
            &\leq |x|+ \|w\||t-s|.
        \end{align*}

 \item [(b)]  $\|DF(t,s,\cdot)\|_{\linf}\leq e^{\|w\||t-s|}$.
        Similarly, 
        \begin{align*}
            |D F(t,s,x)| 
            &= \bigg| I_2 + \int_s^t Dw(F(r,s,x),r)DF(r,s,x)\,dr \bigg|\\
            &\leq 1 + \|w\|\int_s^t |D F(r,s,x)|\,dr.
        \end{align*}
        By the Gr\"onwall's inequality, the estimate holds. Moreover, by the Mean Value Theoreom, also $\Lip F(t,s,\cdot)\leq e^{\|w\||t-s|}$, where $\Lip f$ denotes the global Lipschitz constant of the globally Lipschitz continuous function $f$.
        Observe that we normalise the Fr\"obenius norm for matrices so that $|I_n|=1$. 
        
 \item [(c)] $[F(t,s,\cdot)]_{1,\al}\leq \|w\||t-s|e^{(2+\al)\|w\||t-s|}$.
        With a similar argument, we see that
        \begin{align*}
            |D F(t,s,x)\ -\ & D F(t,s,y)| \\
            &\leq \int_s^t |DF(r,s,x)-DF(r,s,y)||Dw(F(r,s,x))|\,dr\\
                &\qquad + \int_s^t |DF(r,s,y)||Dw(F(r,s,x),r)- Dw(F(r,s,y),r)|\,dr\\
            &\leq  \|w\|\int_s^t|DF(r,s,x)-DF(r,s,y)|\,dr \\
                &\qquad +\int_s^t\|DF(r,s,\cdot)\|_{\linf}[Dw(\cdot,r)]_{0,\al} \Lip F(r,s,\cdot)^{\al}\,dr|x-y|^{\al}\\
            &\leq \|w\|\int_s^t|DF(r,s,x)-DF(r,s,y)|\,dr \\
                &\qquad +\|w\||t-s| e^{(1+\al)\|w\||t-s|}|x-y|^{\al}.
        \end{align*}
        By the Gr\"onwall's lemma, we obtain the estimate above.
\end{itemize}

We present now the estimates on $\ta$, using the representation formula for the solution of the transport equation.
  \begin{itemize}
  	\item [(A)] From the representation formula, it is clear that $\|\ta(\cdot,t)\|_{\linf} \leq \|\ta_0\|$.
  	
  	\item [(B)] In the same way, one proves that $\|\nabla \ta(\cdot,t)\|_{\linf} \leq \|\ta_0\|e^{\|w\|t}$:
  	\begin{align*}
  		|\nabla\ta(x,t)| &= |DF(0,t,x)\nabla\ta_0(F(0,t,x))|\\
  		&\leq \|DF(0,t,\cdot)\|_{\linf}\|\ta_0\|\\
  		&\stackrel{(c)}{\leq} e^{\|w\|t}\|\ta_0\|.
  	\end{align*}
  	
  	\item [(C)] We now focus our attention on the H\"older seminorm and demonstrate that $[\nabla \ta(\cdot,t)]_{0,\al}\leq \|\ta_0\|e^{2\|w\|t}\left(2^{\frac{1-\al}2} + \|w\|te^{\al\|w\|t}\right)$:
  	\begin{align*}
  		|\nabla\ta(x,t)-\nabla\ta(y,t)| &= |DF(0,t,x)\nabla\ta_0(F(0,t,x)) - DF(0,t,y)\nabla\ta_0(F(0,t,y))|\\
  		&\leq |DF(0,t,x)-DF(0,t,y)| |\nabla\ta_0(F(0,t,x))|\\
            &\qquad + |DF(0,t,y)| |\nabla\ta_0(F(0,t,x)) - \nabla\ta_0(F(0,t,y))|\\
        &\leq [DF(0,t,\cdot)]_{0,\al} |x-y|^{\al} \|\ta_0\| \\
            &\qquad + \|DF(0,t,\cdot)\|_{\linf} \Lip \ta_0 \Lip F(0,t,\cdot) |x-y|\\
        &{\leq} \|w\|t e^{(2+\al)\|w\|t}\|\ta_0\| |x-y|^{\al}\\
            &\qquad + e^{2\|w\|t}\|\ta_0\| |x-y|\\
        &\leq \|w\|t e^{(2+\al)\|w\|t}\|\ta_0\| |x-y|^{\al}\\
            &\qquad + 2^{\frac{1-\al}2}e^{2\|w\|t}\|\ta_0\| |x-y|^{\al}\\
        &= \|\ta_0\|e^{2\|w\|t}\left(2^{\frac{1-\al}2} + \|w\|te^{\al\|w\|t}\right)|x-y|^{\al}.
  	\end{align*}
  	
  	\item [(D)] We now can gather together the information from the estimates (A)--(C) to obtain that the first of the estimates in \eqref{app:eq:estimate_ta}: by definition of the H\"older norm (see \cite[\S.~4]{GT:ellipticPDEs}), 
        we have that
        \begin{align*}
        	\|\ta(\cdot,t)\|_{\cal1} &:= \max\{\|\ta(\cdot,t)\|_{\linf}, \|\nabla\ta(\cdot,t)\|_{\linf}\} + [\nabla\ta(\cdot,t)]_{0,\al}\\
        	&\leq\|\ta_0\|e^{\|w\|t} + \|\ta_0\|e^{2\|w\|t}\left(2^{\frac{1-\al}2} + \|w\|te^{\al\|w\|t}\right).
        \end{align*}
    
    \item [(E)] In the following estimates, we look at the H\"older norm of the
        difference $\ta(\cdot,t)-\ta(\cdot,s)$. We start by proving that $\|\ta(\cdot,t)-\ta(\cdot,s)\|_{\linf} \leq \|w\||t-s| \|\ta_0\|$:
        \begin{align*}
            |\ta(x,t)-\ta(x,s)| &= \bigg| \int_s^t \dt \ta(x,r)\,dr\bigg|\\
            &\leq \int_s^t |\dt (\ta_0(F(0,r,x)) )|\,dr\\
            &\leq \int_s^t |\nabla\ta_0(F(0,r,x))| |w(F(0,r,x),0)|\,dr\\
            &\leq \|w\||t-s| \|\ta_0\|.
        \end{align*}
    
    \item [(F)] Here, we look at the distance between the gradients and prove that 
        $\|\nabla\ta(\cdot,t) - \nabla\ta(\cdot,s)\|_{\linf}\leq  2 \|w\| |t-s| \|\ta_0\| e^{\|w\|t}$:
        \begin{align*}
            |\nabla \ta(x,t) - \nabla \ta(x,s)| &=\bigg| \int_s^t \dt \big( DF(0,r,x)\nabla\ta_0(F(0,r,x))\big)\,dr \bigg|\\
            &\leq \int_s^t |\dt (DF(0,r,x)) \nabla\ta_0(F(0,r,x)) \\ 
            &\qquad + DF(0,r,x) D^2\ta_0(F(0,r,x))\dt F(0,r,x)|\,dr\\
            &\leq \int_s^t |Dw(F(0,r,x),0) DF(0,r,x) \nabla\ta_0(F(0,r,x))|\,dr \\
                &\qquad + \int_s^t |DF(0,r,x) D^2\ta_0(F(0,r,x))w(F(0,r,x),0)| \,dr\\
            &\leq \|w\||t-s| \|DF(0,t,\cdot)\|_{\linf} \|\ta_0\| \\
                &\qquad + \|w\||t-s| \|\ta_0\| \|DF(0,t,\cdot)\|_{\linf}\\
            &{\leq}\ 2 \|w\| |t-s| \|\ta_0\| e^{\|w\|t}. 
        \end{align*}

    \item [(G)] The last estimate that we need to conclude the proof is the H\"older seminorm of the
        distance between the gradients, therefore here we demonstrate that $[\nabla\ta(\cdot,t) - \nabla\ta(\cdot,s)]_{0,\al} \leq 2\|\ta_0\| \|w\| e^{\|w\|t} \left(e^{\al\|w\|t} + 2^{\frac{1-\al}2}e^{\|w\|t} + \|w\|te^{(1+\al)\|w\|t}\right)|t-s|$:
        \begin{align*}
            |\nabla \ta(x,t) - \nabla \ta(x,s) - \nabla \ta(y,t) + \nabla \ta(y,s)| 
            &= \bigg| \int_s^t \dt [\nabla \ta(x,r) - \nabla \ta(y,r)] \,dr \bigg|. 
        \end{align*}
        We start studying $\dt [\nabla \ta(x,r) - \nabla \ta(y,r)]$:
        \begin{align*}
            \dt [\nabla \ta(x,r) -& \nabla \ta(y,r)] = \dt \left[ DF(0,r,x)\nabla\ta_0(F(0,r,x)) -  DF(0,r,y)\nabla\ta_0(F(0,r,y))\right]\\
            &= Dw(F(0,r,x),0)DF(0,r,x) \nabla\ta_0(F(0,r,x)) \\
                &\qquad +DF(0,r,x) D^2\ta_0(F(0,r,x)) w(F(0,r,x),0)\\
                &\qquad -Dw(F(0,r,y),0)DF(0,r,y) \nabla\ta_0(F(0,r,y)) \\
                &\qquad -DF(0,r,y) D^2\ta_0(F(0,r,y)) w(F(0,r,y),0)\\
            &= \left\{Dw(F(0,r,x),0) - Dw(F(0,r,y),0)\right\} DF(0,r,x) \nabla\ta_0(F(0,r,x))\\
                &\qquad + Dw(F(0,r,y),0) \left\{DF(0,r,x) -DF(0,r,y) \right\} \nabla\ta_0(F(0,r,x))\\
                &\qquad + Dw(F(0,r,y),0) DF(0,r,y) \left\{\nabla\ta_0(F(0,r,x)) -\nabla\ta_0(F(0,r,y)) \right\}\\
                &\qquad + \left\{ DF(0,r,x) -DF(0,r,y) \right\} D^2\ta_0(F(0,r,x)) w(F(0,r,x),0)\\
                &\qquad + DF(0,r,y) \left\{ D^2\ta_0(F(0,r,x)) - D^2\ta_0(F(0,r,y))\right\} w(F(0,r,x),0)\\
                &\qquad + DF(0,r,y) D^2\ta_0(F(0,r,y)) \left\{w(F(0,r,x),0) - w(F(0,r,y),0)\right\}.
        \end{align*}
        Therefore, when we consider the modulus of the expression above,
        \begin{align*}
            |\dt [\nabla \ta(x,r) -& \nabla \ta(y,r)]| \\
            &\leq |Dw(F(0,r,x),0) - Dw(F(0,r,y),0)| |DF(0,r,x)| |\nabla\ta_0(F(0,r,x))|\\
                &\qquad + |Dw(F(0,r,y),0)| |DF(0,r,x) -DF(0,r,y)| |\nabla\ta_0(F(0,r,x))|\\
                &\qquad + |Dw(F(0,r,y),0)| |DF(0,r,y)| |\nabla\ta_0(F(0,r,x)) -\nabla\ta_0(F(0,r,y))|\\
                &\qquad + |DF(0,r,x) -DF(0,r,y)| |D^2\ta_0(F(0,r,x))| |w(F(0,r,x),0)|\\
                &\qquad + |DF(0,r,y)| |D^2\ta_0(F(0,r,x)) - D^2\ta_0(F(0,r,y))| |w(F(0,r,x),0)|\\
                &\qquad + |DF(0,r,y)| |D^2\ta_0(F(0,r,y))| |w(F(0,r,x),0) - w(F(0,r,y),0)|\\
            &\leq [Dw(\cdot,0)]_{0,\al} \Lip F(0,r,\cdot)^{\al} |x-y|^{\al} \|DF(0,r,\cdot)\|_{\linf} \|\ta_0\|\\
                &\qquad + \|Dw(\cdot,0)\|_{\linf}  [DF(0,r,\cdot)]_{0,\al} |x-y|^{\al} \|\ta_0\|\\
                &\qquad + \|Dw(\cdot,0)\|_{\linf} \|DF(0,r,\cdot)\|_{\linf} \Lip \nabla\ta_0 \Lip F(0,r,\cdot) |x-y|\\
                &\qquad + [DF(0,r,\cdot)]_{0,\al} |x-y|^{\al} {\color{black}\|D^2\ta_0\|_{\linf}} \|w(\cdot,0)\|_{\linf}\\
                &\qquad + \|DF(0,r,\cdot)\|_{\linf} {\color{black}[D^2\ta_0]_{0,\al}}\Lip F(0,r,\cdot)^{\al}|x-y|^{\al} \|w(\cdot,0)\|_{\linf}\\
                &\qquad + \|DF(0,r,\cdot)\|_{\linf} {\color{black}\|D^2\ta_0\|_{\linf}} \Lip w(\cdot,0) \Lip F(0,r,\cdot) |x-y|\\
            &\leq \|w\| e^{(1+\al)\|w\|r} |x-y|^{\al}  \|\ta_0\|\\
                &\qquad + \|w\|^2 r e^{(2+\al)\|w\|r} |x-y|^{\al} \|\ta_0\|\\
                &\qquad + \|w\| e^{2\|w\|r} \|\ta_0\| |x-y|\\
                &\qquad + \|w\|^2 r e^{(2+\al)\|w\|r} |x-y|^{\al} \|\ta_0\| \\
                &\qquad + \|w\|e^{(1+\al)\|w\|r} \|\ta_0\| |x-y|^{\al} \\
                &\qquad + \|w\|e^{2\|w\|r} \|\ta_0\| |x-y|\\
            &= 2\|\ta_0\| \|w\| e^{\|w\|r} \left(e^{\al\|w\|r} + 2^{\frac{1-\al}2}e^{\|w\|r} + \|w\|re^{(1+\al)\|w\|r}\right)|x-y|^{\al}.
        \end{align*}
        Observe that we require $\ta_0\in\cal2(\T^2)$: the estimate would not work for initial data in $\cal1(\T^2)$.
        All the quantities in the terms above are increasing in $r$, therefore we estimate 
        \begin{align*}
            |\nabla \ta(x,t) - \nabla &\ta(x,s) - \nabla \ta(y,t) + \nabla \ta(y,s)| 
            \leq \int_s^t |\dt [\nabla \ta(x,r) - \nabla \ta(y,r)]| \,dr\\
            &\leq 2\|\ta_0\| \|w\| e^{\|w\|t} \left(e^{\al\|w\|t} + 2^{\frac{1-\al}2}e^{\|w\|t} + \|w\|te^{(1+\al)\|w\|t}\right)|t-s||x-y|^{\al}.
        \end{align*}

    \item [(H)] By combining the estimates (E)--(G), one obtains the estimate on the full H\"older norm:
        $\|\ta(\cdot,t) - \ta(\cdot,s)\|_{\cal1} \leq 2\|\ta_0\| \|w\| e^{\|w\|t} \left(1+e^{\al\|w\|t} + 2^{\frac{1-\al}2}e^{\|w\|t} + \|w\|te^{(1+\al)\|w\|t}\right)|t-s|$.
        
        By definition of the H\"older norm, we have that
        \begin{align*}
            \|\ta(\cdot,t) - \ta(\cdot,s)\|_{\cal1} 
                &:= \max\{ \|\ta(\cdot,t) - \ta(\cdot,s)\|_{\linf}, 
                    \|\nabla\ta(\cdot,t) - \nabla\ta(\cdot,s)\|_{\linf}\}
                    + [\ta(\cdot,t) - \ta(\cdot,s)]_{1,\al}\\
                &\leq 2\|\ta_0\| \|w\| e^{\|w\|t} \left(1+e^{\al\|w\|t} + 2^{\frac{1-\al}2}e^{\|w\|t} + \|w\|te^{(1+\al)\|w\|t}\right)|t-s|.
        \end{align*}

    \item [(I)] Finally, we can prove the second estimate in \eqref{app:eq:estimate_ta}.
    By the estimates (D) and (H), one has that 
    \begin{align*}
    	\|\ta\|_{\cbe0_t\cal1_x} 
    	&:= \sup_{t\in\I}\|\ta(\cdot,t)\|_{\cal1} 
            + \sup_{\substack{t,s\in\I \\ t\neq s}} \frac{\|\ta(\cdot,t)-\ta(\cdot,s)\|_{\cal1}}{|t-s|^{\be}}\\
        &\leq  \|\ta_0\|e^{\|w\|\tau} \left(1 + 2^{\frac{1-\al}2}e^{\|w\|\tau} + \|w\|\tau e^{(1+\al)\|w\|\tau}\right)\\
            &\qquad + 2\|\ta_0\| \|w\|\tau^{1-\be} e^{\|w\|\tau} \left(1+e^{\al\|w\|\tau} + 2^{\frac{1-\al}2}e^{\|w\|\tau} + \|w\|\tau e^{(1+\al)\|w\|\tau}\right).
    \end{align*}
  \end{itemize}
  This concludes the proof.
  \end{proof}

  In a similar way, the calculations above can be computed
  for $w\in\cbe0(\I;\cal{k+1}_{\s}(\T^2;\R^2))$ with $k\in\N$,
  using 
 Fa\`a di Bruno's formula for higher order chain rule: for $k\in\N$
    \begin{align*}
    	D^k\ta(x,t)\ = \sum_{\substack{m\in\N^k \\ \sum_{n=1}^k nm_n=k}} \frac{k!}{\prod_{n=1}^k m_n! n!^{m_n}}
        D^{\sum_{n=1}^k m_n} \ta_0(F(0,t,x)) \prod_{n=1}^k \left(D^n F(0,t,x)\right)^{m_n},
    \end{align*}
    where for any $n\in\N$
 \begin{align*}
 	D^n F(t,0,x) \ =\sum_{\substack{m\in\N^n \\ \sum_{j=1}^n jm_j=n}} \frac{n!}{\prod_{j=1}^n m_j! j!^{m_j}}
 	\int_0^t D^{\sum_{j=1}^n m_j}w(F(r,0,x),r)\prod_{j=1}^n \left(D^j F(r,0,x)\right)^{m_j}.
 \end{align*}
 As the calculation is tedious, but the same argument as above needs to be applied, we do not include the details here.


\section*{Acknowledgment}
We would like to thank Jacques Vanneste for suggesting the problem under study in this article. We would also like to thank Jacques Vanneste and Beatrice Pelloni for several stimulating conversations related to the work herein. Finally, we are grateful to Miles Wheeler, Martin Dindo\v{s} and Heiko Gimperlein for discussions on elliptic problems on unbounded domains.
  
S. Lisai was supported by The Maxwell Institute Graduate School in Analysis and its
Applications (MIGSAA) funded by the UK Engineering and Physical
Sciences Research Council (grant EP/L016508/01), the Scottish Funding Council, Heriot-Watt
University and the University of Edinburgh.
M. Wilkinson is supported by the EPSRC Standard Grant EP/P011543/1.
  

\addcontentsline{toc}{section}{References}
\bibliographystyle{alpha}
\bibliography{biblio}

\end{document}